\documentclass[12 pt,oneside]{amsart}

\usepackage[T1]{fontenc}
\usepackage[utf8]{inputenc}

\usepackage{xypic}
\usepackage{amsmath}
\usepackage{mathdesign}
\usepackage{hyperref}

\usepackage{enumerate}

\usepackage{amssymb}
\usepackage{dsfont}
\usepackage{tikz-cd}

\usepackage{extsizes}
\usepackage{capt-of}

\usepackage{lipsum}
\usepackage{footnote}
\usepackage{graphicx}

\usepackage{bbm}

\usepackage{float}

\xyoption{all}

\newif\iffurther
\furthertrue

\newtheorem{thm}{Theorem}[section]

\newtheorem{cor}[thm]{Corollary}

\newtheorem{lem}[thm]{Lemma}
\newtheorem{prop}[thm]{Proposition}

\newtheorem{ques}[thm]{Question}

\newtheorem{fact}[thm]{Fact}

\def\[{\left[}
\def\]{\right]}

\def\wt{\varphi}
\def\fdim{{\operatorname{fdim}}}

\def\GK{{\operatorname{GKdim}}}

\def\Span{{\operatorname{Span}}}

\def\ind{{\operatorname{\mathbbm{1}}}}
\def\l{{\operatorname{length}}}

\long\def\forget#1\forgotten{{}}

\title[Complexity and recurrence in words and related structures]{Complexity and recurrence in infinite words and related structures}

\author{Be'eri Greenfeld} \address{Department of Mathematics, University of Washington, Seattle, WA 98195, USA} \email{grnfld@uw.edu}

\author{Carlos Gustavo Moreira} \address{SUSTech International Center for Mathematics, Shenzhen, Guangdong,
People’s Republic of China;
IMPA, Estrada Dona Castorina 110, 22460-320, Rio de Janeiro, Brazil} \email{gugu@impa.br}

\author{Efim Zelmanov} \address{SUSTech International Center for Mathematics, Shenzhen, Guangdong,
People’s Republic of China} \email{efim.zelmanov@gmail.com}

\begin{document}

\begin{abstract}
We study the asymptotics and fine-scale behavior of quantitative combinatorial measures of infinite words and related dynamical and algebraic structures. 

We construct infinite recurrent words $w$ whose complexity functions $p_w(n)$ are arbitrarily close to linear, but whose discrete derivatives are not bounded from above by $p_w(n)/n$. Moreover, we construct words of polynomially bounded complexity whose discrete derivatives exceed $p_w(n)/n^\varepsilon$ infinitely often, for every given $\varepsilon>0$. These provide negative answers in a strong sense to an open question of Cassaigne from 1997, showing that his theorem on words of linear complexity is best possible.

Next, we characterize, up to a linear multiplicative error, the complexity functions of strictly ergodic subshifts, showing that every non-decreasing, submultiplicative function arises in this setting. This gives the first `industrial' construction of strictly ergodic subshifts of prescribed subexponential complexity.

We then investigate quantitative recurrence in uniformly recurrent words and, as an application, address a question of Bavula from 2006 related to holonomic inequalities on the spectrum of possible filter dimensions of simple associative algebras: we construct simple algebras of prescribed filter dimension in $[1,\infty)$ and essentially settling the problem entirely in the graded case. 
Throughout, we construct uniformly recurrent words of linear complexity and with arbitrary polynomial recurrence growth.
\end{abstract}

\maketitle

\section{Introduction}

\subsection{Infinite words and their complexity}
Let $\Sigma$ be a finite alphabet and let
$$
w=w_1w_2\cdots \in \Sigma^{\mathbb{N}}
$$
be an infinite word over $\Sigma$. The \emph{complexity function} of $w$
$$
p_w(n) := \# \{\text{Subwords of } w \ \text{of length } n\}
$$
counts the number of subwords (also called `factors') of $w$ of length $n$. The complexity function of an infinite word (or more generally, of a subshift---that is, a symbolic dynamical system under the shift maps $\Sigma^{\mathbb{N}}\to \Sigma^{\mathbb{N}}$ or $\Sigma^{\mathbb{Z}}\to \Sigma^{\mathbb{Z}}$) is a fundamental measure of `how complicated' the word is. Often, the complexity functions of infinite words reflect interesting properties or phenomena of processes from various mathematical contexts from which these words arise. For instance, infinite words of bounded complexity are precisely the eventually periodic words; words of the slowest possible unbounded complexity (namely, $p_w(n)=n+1$) are obtained from irrational rotations on the circle, and can be thought of as `1D quasi-crystals' \cite{MH}; and irrational numbers whose decimal (or any base) expansion has linear complexity are transcendental \cite{AB07}. Complexity functions of infinite words, and more generally of hereditary formal languages, are related to the growth of groups, semigroups, and associative algebras; see, e.g., \cite{BZ,BBL,Gri1,Gromov,Gromov2,KL,SB,Tro}.

\medskip

How do complexity functions of infinite words look? Let $w$ be an infinite word. Then $p_w(n)$ is non-decreasing, and in fact, either bounded or increasing (the Morse-Hedlund theorem), and submultiplicative. 
A longstanding open problem of fundamental importance in combinatorics on words and symbolic dynamics is the `inverse problem': to characterize the possible complexity functions of infinite words (see \cite{All94, Cas2, Cas, Cas3,Fer96,Fer99,GMZ,MM1,MM2} and references therein for the question and its variations, as well as for realizations of interesting families of functions as complexity functions of words). Conceptually, any additional feature of complexity functions should be grounded in some combinatorial reasoning that reveals a fundamental part of the nature of combinatorics of infinite words. In \cite{GMZ} we proved that, up to the notion of asymptotic equivalence\footnote{For non-decreasing functions $f,g\colon \mathbb{N}\rightarrow \mathbb{N}$, we say $f\preceq g$ if there exist constants $C,D>0$ such that $f(n)\leq Cg(Dn)$ for all $n\gg 1$; and $f\sim g$ if $f\preceq g$ and $g\preceq f$.} from large-scale geometry, the above conditions are the only ones. While playing an essential role in large-scale geometry and geometric group theory, this notion of equivalence is rather coarse and, for instance, identifies all functions that are $\Theta(n)$.

\subsection{Discrete derivatives of complexity functions and Cassaigne's question}
When considering complexity functions not up to asymptotic equivalence but numerically, more properties and obstacles to realizing a function as the complexity function of an infinite word appear.
Given a function $f\colon \mathbb{N}\to \mathbb{N}$, its \emph{discrete derivative} is defined by
$$
f'(n):=f(n)-f(n-1)
$$
(with $f'(0)=0$). The discrete derivative is a finer invariant, not preserved under asymptotic equivalence. In \cite{Cas3}, Cassaigne proved that if $p_w(n)=O(n)$ then $p_w'(n)=O(1)$, affirming a conjecture of Ferenczi and providing one of the only known obstacles preventing a function from describing the complexity of an infinite word. Cassaigne further asked whether this can be generalized as follows: if $p_w(n)$ is polynomially bounded, must $p_w'(n)=O\left(\frac{p_w(n)}{n}\right)$? 
Note that this phenomenon can indeed occur only for words of polynomially bounded complexity (Lemma \ref{lem:calculus} below). We show that the phenomenon in Cassaigne's Theorem is really unique to the linear complexity setting, answering the above question in the negative, in the following strong senses. First,

\begin{thm} \label{thm:Cassaigne}
    Let $g\colon \mathbb{N}\rightarrow \mathbb{N}$ be any superlinear function, i.e., $\frac{g(n)}{n}\xrightarrow{n\rightarrow \infty}~\infty$. Then there exists a recurrent infinite word $w\in \Sigma^{\mathbb{N}}$ over a finite alphabet $\Sigma$ such that $p_w(n)\leq g(n)$ for all $n\gg 1$ and
    $$
    p_w'(n)\neq O\left(\frac{p_w(n)}{n}\right).
    $$
\end{thm}

We next show that it is possible, even within the realm of polynomially bounded complexity, that the discrete derivative $p_w'(n)$ in fact gets very close to $p_w(n)$:

\begin{thm} \label{thm:Cassaigne2}
For every $\varepsilon>0$ there exists an infinite word $w$ over a finite alphabet $\Sigma$ such that $p_w(n)$ is polynomially bounded and
$$
p_w'(n)\geq \frac{p_w(n)}{n^\varepsilon} \quad \text{for infinitely many } n.
$$
In particular,
$$
p_w'(n)\neq O\left(\frac{p_w(n)}{n^\varepsilon}\right).
$$
\end{thm}

\subsection{Complexity of strictly ergodic subshifts} 
A dynamical system $(\mathcal{X},T)$ where $\mathcal{X}$ is a compact topological space and $T\colon \mathcal{X}\rightarrow \mathcal{X}$ is a homeomorphism is \emph{uniquely ergodic} if it admits a unique invariant Borel probability measure. A minimal dynamical system which is also uniquely ergodic is called \emph{strictly ergodic}. In \cite{G1,G2,GS} strictly ergodic subshifts of prescribed entropy have been constructed. We focus on $\mathbb{Z}$-subshifts over some finite alphabet $\Sigma$.

In \cite{GMZ}, we provided an asymptotic characterization of the complexity functions of infinite words and recurrent words, hence of $\mathbb{Z}$-subshifts, and gave a similar realization theorem for complexity functions of minimal subshifts (equivalently, uniformly recurrent words) up to an error factor.

We realize every ``natural candidate'' function --- namely, non-decreasing and submultiplicative --- as the complexity function of some \emph{strictly ergodic} $\mathbb{Z}$-subshift, up to a linear error factor. To the best of our knowledge, this is the first `industrial' construction of zero-entropy strictly ergodic subshifts of prescribed complexity.

\begin{thm} \label{thm:strictly}
Let $f\colon \mathbb{N}\rightarrow \mathbb{N}$ be a non-decreasing, submultiplicative function. Then there exists a strictly ergodic $\mathbb{Z}$-subshift $\mathcal{X}$ such that 
$$
f(n)\preceq p_\mathcal{X}(n)\preceq n f(n).
$$
\end{thm}

Notice that in the above form, one cannot relax the linear error factor in the upper bound, since for any sub-linear, unbounded function $f$, there is no subshift whose complexity is $\sim f$.

\subsection{Quantitative recurrence, infinite-dimensional simple algebras and Bavula's question}
In the last part, we turn to quantify \emph{how recurrent a uniformly recurrent word can be}.
Recall that $w$ is uniformly recurrent if for every finite factor $u$ of $w$ there exists some $K_u$ such that every factor of $w$ of length $K_u$ contains an occurrence of $u$. The \emph{recurrence function} of $w$ is a quantitative measure of the efficiency of recurrence in $w$:
$$
Rec_w(n) = \min\Big\{K\in \mathbb{N}\ \colon\ \substack{\text{every factor of}\ w\ \text{of length}\ K\ \text{contains} \\ \text{an occurrence of every factor of}\ w\ \text{of length}\ n} \Big\}
$$
(see \cite{CasRec,CC06} and references therein.)

\begin{prop} \label{prop:arbitrary rec}
Let $\alpha\geq 1$ be a real number. Then there exists a uniformly recurrent word $w$ whose complexity function is linear and whose recurrence function has degree of polynomial growth $\alpha$, namely,
$$
p_w(n) \asymp n
\quad\text{and}\quad
\overline{\lim\limits_{n\rightarrow \infty}} \log_n Rec_w(n) = \alpha.
$$
\end{prop}

The words in Proposition \ref{prop:arbitrary rec} are constructed using a carefully chosen, iterative sequence of substitutions of finite words which converge to a uniformly recurrent word with the desired properties.

We apply this construction to address an open question of Bavula in noncommutative algebra, related to the notion of \emph{filter dimension} of infinite-dimensional simple algebras. 
This notion, defined and studied by Bavula and others \cite{Bav09,Bavula,Bav05,Bav98,Bav00,Bav08}, arises naturally in the context of holonomicity, generalizing Bernstein's Inequality on the dimensions of $D$-modules. In some sense, the filter dimension measures ``how effective a Bernstein-type inequality can be'' for a given simple algebra. The filter dimension is related to quantum completely integrable systems, to the study of commutative subalgebras of noncommutative algebras \cite{Bav05}, and to the Krull and Gel'fand-Kirillov (GK) dimensions of modules. 
One of the most fundamental questions related to the filter dimension is:
\begin{ques}[{\cite{Bavula}}]
Which numbers are the possible values of the filter dimension?
\end{ques}
For instance, Bavula \cite{Bav98} proved that if $A$ is a differential polynomial ring of a smooth variety, then $\fdim(A)=1$. 
We prove:

\begin{thm} \label{thm:main Bavula}
Let $\mathbbm{k}$ be an arbitrary field. For every real number $\alpha\geq 1$ there exists a finitely generated, $\mathbb{Z}$-graded, simple $\mathbbm{k}$-algebra $A$ of GK-dimension $2$ whose filter dimension is $\alpha$.
\end{thm}

This gives the first examples of simple algebras with arbitrary filter dimension, and essentially settles the problem for non-trivially $\mathbb{Z}$-graded simple algebras (see Proposition \ref{prop:prop}).

\section{Conventions and preliminaries}

Let $\Sigma$ be a finite alphabet. Let $w\in \Sigma^{\mathbb{N}}$ be an infinite word; we occasionally think of \emph{bi-infinite} words $w\in \Sigma^{\mathbb{Z}}$. We denote the $i$-th letter of $w$ by $w_i$ or $w[i]$, indexing $i=1,2,\dots$ (or $i\in \mathbb{Z}$ in the bi-infinite case) and $w[i,j]=w[i]\cdots w[j]$. 
We let $L_w(n)$ be the set of all length-$n$ factors of $w$, namely, $$L_w(n)=\{u\in \Sigma^n\ |\ \exists i:\ w[i,i+n-1]=u\}.$$ The set $L_w = \bigcup_{n\geq 0} L_w(n)$ is the \emph{hereditary language} (or factorial language) associated with $w$. The complexity function of $w$ is $p_w(n)=|L_w(n)|$.

We say that $w$ is \emph{recurrent} if for every finite factors $u_1,u_2\in L_w$ there exists some word $v$ such that $u_1vu_2\in L_w$. When $w$ is a one-sided infinite word, this is equivalent to the assertion that every finite subword appears infinitely many times.
We say that $w$ is \emph{uniformly recurrent} if for every finite factor $u \in L_w$ there exists some $K_u$ such that every factor of $w$ of length $K_u$ contains an occurrence of $u$.

Let $T\colon \Sigma^{\mathbb{N}}\rightarrow \Sigma^{\mathbb{N}}$ or $T\colon \Sigma^{\mathbb{Z}}\rightarrow \Sigma^{\mathbb{Z}}$ be the shift map given by $T(w)[i]=w[i+1]$. A \emph{subshift} is a non-empty, closed, shift-invariant subset of $\Sigma^{\mathbb{N}}$ or $\Sigma^{\mathbb{Z}}$.
Let $w\in \Sigma^{\mathbb{Z}}$ be a bi-infinite word over $\Sigma$. Let $$\mathcal{X}_w=\overline{\{T^i(w)\colon i\in \mathbb{Z}\}}\subseteq \Sigma^{\mathbb{Z}}$$ be the $\mathbb{Z}$-subshift generated by $w$. Then $w$ is uniformly recurrent if and only if $\mathcal{X}_w$ is minimal, namely, contains no proper subshifts.

\section{Infinite words of super-linear complexity}

In this section we prove Theorem \ref{thm:Cassaigne}.

\begin{lem} \label{lem:calculus}
    Let $f\colon \mathbb{N}\rightarrow \mathbb{N}$ be a non-decreasing, submultiplicative function. If $f'(n) = O\left(\frac{f(n)}{n}\right)$ then for every $K \geq 1$ there exists some $L$ such that $f(Kn)\leq Lf(n)$ for all $n$. In particular, $f$ is polynomially bounded.
\end{lem}

\begin{proof}
Suppose that there exists some $C>0$ such that $f'(n)\leq C\frac{f(n)}{n}$ for all $n$. Then 
$$f(n+1)-f(n)=f'(n+1)\leq C\frac{f(n+1)}{n+1} \leq D\frac{f(n)}{n}$$
where $D$ can be taken to be $C\cdot f(1)$. Hence
$$f(n+1) \leq \left(1+\frac{D}{n}\right) \cdot f(n)\leq \exp\left(\frac{D}{n}\right) \cdot f(n)$$
and, by induction,
$$f(n+i)\leq \exp\left(\frac{iD}{n}\right) \cdot f(n)$$ for every $i\geq 0$. In particular, taking $i=(K-1)n$, we obtain
$$f(Kn) \leq \underbrace{\exp((K-1)D)}_{=: L}\cdot f(n),$$
as claimed.

In particular, there exists some $L>0$ such that $f(2n)\leq Lf(n)$ for all $n$, and thus $$f(2^t)\leq L^tf(1)=f(1)\cdot (2^t)^{\log_2 L}.$$  Hence for each $n$, we have $$f(n)\leq f(2^{\lceil \log_2 n\rceil})\leq f(1)\cdot (2^{\lceil \log_2 n\rceil})^{\log_2 L} \leq  f(1)\cdot (2n)^{\log_2 L}$$ so $f$ is polynomially bounded.
\end{proof}

Let $g\colon \mathbb{N}\rightarrow \mathbb{N}$ be a superlinear function, namely, $\frac{g(n)}{n}\xrightarrow{n\rightarrow \infty} \infty$. 
We fix a sequence of integers $$1<d_2<d_3<\cdots$$ such that $d_{i+1}>4d_i$ for all $i\geq 2$ and such that $\omega(n)  := \max\{i\ \colon\ 2d_i\leq n \}$ satisfies $$\omega(n)! < \frac{g(n)}{2(n+1)}$$ (we think of $\max \emptyset$ as $0$ for the purpose of this definition) for $n\gg 1$. This is indeed possible since $\frac{g(n)}{2(n+1)}\xrightarrow{n\rightarrow \infty} \infty$. Furthermore, we may assume that $d_2,d_3,\dots$ are all powers of $2$.

Let $f(1)=2$ and for every $n\geq 2,\ n\neq 2d_2,2d_3,\dots$ let $f(n)=f(n-1)+1$; for $n=2d_i$, put $f(2d_i)=if(d_i)$.

\begin{prop} \label{prop:properties}
    The function $f$ above satisfies (i) $f(n)<f(n+1)$ and (ii) $f(2n)\leq f(n)^2$ for all $n\in \mathbb{N}$, and $f(n) \leq g(n)$ for all $n\gg 1$.
\end{prop}

\begin{proof}
First, we prove by induction that $f'(n)\geq 1$ for all $n\geq 2$. If $n\geq 2,\ n\neq 2d_i$ then $f'(n)=1$ by definition. For $n=2d_i$, notice that \begin{eqnarray*} f'(2d_i) & = & if(d_i)-f(2d_i-1) \\ & = & if(d_i)-(f(d_i)+d_i-1) \\ & = & (i-1)f(d_i)-(d_i-1)  \geq  f(d_i)-(d_i-1) \geq 1 \end{eqnarray*} since by induction we may assume that $f'(m)\geq 1$ for all $m<2d_i$.

We now prove by induction that $f(2n)\leq f(n)^2$ for all $n$. If $n\in [1,d_2-1]$ then $$f(2n)=2n+1\leq (n+1)^2=f(n)^2.$$ If $n=d_i$ then $f(2d_i)=if(d_i)\leq d_if(d_i)\leq f(d_i)^2$. Now suppose that there exists some $i\geq 2$ such that $n\in [d_i+1,d_{i+1}-1]$. Then 
\begin{eqnarray*} 
f(n)^2 & = & \left( f(d_i) + (n - d_i) \right)^2 \\ & = & f(d_i)^2+2(n-d_i)+(n-d_i)^2 \\ & \geq & d_i f(d_i) + 2(n-d_i) \\ & \geq & i f(d_i)+2(n-d_i) = f(2d_i)+2n-2d_i = f(2n). \end{eqnarray*}

Finally, notice that if $n\neq 2d_i$ for all $i$ then $$\frac{f(n)}{f(n-1)}=1+\frac{1}{f(n-1)}\leq 1+\frac{1}{n}=\frac{n+1}{n},$$ and if $n=2d_i$ for some $i$ then $$\frac{f(n)}{f(n-1)}\leq \frac{f(2d_i)}{f(d_i)} = i.$$
Therefore, for $n\gg 1$,
\begin{eqnarray*}
    f(n) & = & f(1)\cdot \frac{f(2)}{f(1)} \cdots \frac{f(n)}{f(n-1)} \\ & \leq & 2 \cdot \left( \prod_{k=2}^{n} \frac{f(k)}{f(k-1)} \right) \cdot \left( \prod_{i:\ 2d_i\leq n} \frac{f(2d_i)}{f(2d_i-1)} \right) \\ & \leq & 2 \cdot \left( \prod_{k=2}^{n} \frac{k+1}{k} \right) \cdot\left( \prod_{i\leq \omega(n)} i \right) \leq 2(n+1)\cdot \omega(n)! < g(n),
\end{eqnarray*}
as desired.
\end{proof}

\begin{proof}[{Proof of Theorem \ref{thm:Cassaigne}}]
Let $f\colon \mathbb{N}\rightarrow \mathbb{N}$ be as above. By Proposition \ref{prop:properties} and \cite[Theorem~1.1]{GMZ}, $f$ is asymptotically equivalent to the complexity function of some recurrent infinite word $w$, namely, there exist some constants $C_1,C_2,D_1,D_2>0$ such that
$$f(n)\leq C_1p_w(D_1n)\ \text{and}\ p_w(n)\leq C_2f(D_2n)$$
for all $n \gg 1$. Since $f,p_w$ are non-decreasing, we may enlarge $D_2$ and assume that it is a power of $2$. We claim that $p_w'(n)\neq O\left(\frac{p_w(n)}{n}\right)$. Assume otherwise. By Lemma \ref{lem:calculus}, there exists some $L$ such that $p_w(2 D_1 D_2 n)\leq Lp_w(n)$ for all $n$. Now $$f(2D_2 n)\leq C_1p_w(2D_1D_2 n) \leq L C_1 p_w(n) \leq L C_1 C_2 f(D_2 n).$$
Pick $i\gg 1$ such that $i > L C_1 C_2$ and $d_i > D_2$. Take $n=\frac{d_i}{D_2}$ -- an integer, since $d_i>D_2$ are both powers of $2$ -- and observe that
$$f(2d_i) = f(2D_2 n) \leq L C_1 C_2 f(d_i) < i f(d_i),$$
contradicting the way $f(x)$ is defined for $x=2d_i$. The proof is completed.
\end{proof}

\section{Discrete derivatives of polynomially bounded complexity functions}

We start by describing the construction, which falls into Case I of our general construction from \cite{GMZ}. This is not surprising, since the main result of \cite{GMZ} realizes \emph{any} increasing submultiplicative function as the complexity functions of an infinite word, up to asymptotic equivalence. However, the proof of Theorem \ref{thm:Cassaigne2} cannot be simply derived as a corollary from \cite{GMZ}, since discrete derivatives are very `badly behaved' with respect to asymptotic equivalence. Instead, we must delve into a much more precise construction and thoroughly analyze its complexity function, and its discrete derivative.

\subsection{Construction}
We work over the alphabet $\Sigma = \{0,1,2\}$. We denote the set of all finite words over $\Sigma$ by $\Sigma^*$. Fix some integer $r \geq 2$, to be determined in the sequel. Inductively define the following data:

\begin{align} \label{eq:construction}
 & \{n_k\}_{k=1}^{\infty}, \ \{s_k\}_{k=1}^{\infty}, \ \text{sequences of positive integers} \\ & X_k\subseteq \Sigma ^{n_k} \ \text{such that} \ |X_k|=s_k,  \nonumber \\ & \{k_l\}_{l=1}^{\infty}, \ \text{an increasing sequence of positive integers} \nonumber \\ & \ \ \ \ \ \ \ \ \ \ \ \ \text{such that} \ s_{k_l}=n_{k_l}^\alpha\ \text{where}\ \alpha=\frac{\log 4}{\log 3}.  \nonumber
\end{align}
First, $n_k=3^{k-1}$ for all $k\geq 1$. We let $s_1=2$ and $X_1=\{1,2\}$; and $$s_2=4,\ X_2=X_1 0 X_1=\{101,102,201,202\}.$$ We let
$k_1=2$, and observe that $s_{k_1}=s_2=4=3^{\alpha} =n_2^{\alpha}=n_{k_1}^{\alpha}$.

Given $l\geq 1$ such that $s_{k_l}=n_{k_l}^{\alpha}$, we define $k_{l+1}:=(k_l-1)2^r+1$. We observe that, since $r\geq 2$, then $2^r\geq r+2$ and so $$k_l \cdot (2^r - 1)\geq 2^{r+1}-2\geq 2^r+r,$$ hence $$k_{l+1} = (k_l - 1) 2^r+1\geq k_l+r+1.$$ We then define the rest of data (\ref{eq:construction}) as follows:
\begin{itemize}
    \item For $0\leq j\leq r-1$, let $$s_{k_l+j+1}=s_{k_l+j}^2\ \ \text{and}\ \  X_{k_l+j+1}=X_{k_l+j} 0^{n_{k_l+j}} X_{k_l+j}.$$ Observe that $$|X_{k_l+j+1}| = |X_{k_l+j}|^2=s_{k_l+j}^2=s_{k_l+j+1}$$ and $$s_{k_l+r} = s_{k_l}^{2^r} = n_{k_l}^{\alpha 2^r} = 3^{(k_l-1)\cdot \alpha 2^r}.$$
    \item For $k_l+r+1\leq m\leq k_{l+1}$, let $s_m = s_{k_l+r}$. Enumerate $$X_{m-1}=\{\alpha_1,\dots,\alpha_{s_{k_l+r}}\}$$ and let $$X_m=\{\alpha_i 0^{n_{m-1}}\alpha_{i+1 \mod s_{k_l+r}}\ |\ i=1,\dots,s_{k_l+r}\}$$
    where we interpret $x\mod s_{k_l+r}\in [1,s_{k_l+r}]$. Observe that $$|X_m|=s_{k_l+r}=s_m$$ and $$s_{k_{l+1}} = s_{k_l+r} = 3^{(k_l-1)\cdot \alpha 2^r} = \left(3^{(k_l-1)2^r}\right)^\alpha = \left(3^{k_{l+1} - 1}\right)^\alpha = n_{k_{l+1}}^{\alpha}.$$
\end{itemize}
This completes the inductive definition of data (\ref{eq:construction}).

\begin{lem} \label{lem:11}
\begin{enumerate} 
    \item For every $k\geq 1$, we have $$X_k\subseteq 1\Sigma^*1\cup 1\Sigma^*2\cup 2\Sigma^*1\cup 2\Sigma^*2.$$ 
    \item Let $k < k'$. For every $\alpha\in X_k$ there exist some $\beta,\gamma\in X_{k'}$ such that $\alpha 0^{n_k}$ is a prefix of $\beta$ and $0^{n_k}\alpha$ is a suffix of $\gamma$.
    \item If $\alpha$ is a prefix (resp., suffix) of a word from $X_k$ for some $k$, and $|\alpha|\leq n_d = 3^{d-1}$ then $\alpha$ is a prefix (resp., suffix) of a word from $X_{d}$.
\end{enumerate}
\end{lem}

\begin{proof}
\begin{enumerate}
    \item This follows by induction, since $X_1=\{1,2\}$ and $X_{k+1}\subseteq X_k 0^{n_k} X_k$ for all $k\geq 1$.
    \item It suffices to prove the claim in case that $k'=k+1$. By construction, either $X_{k+1}=X_k 0^{n_k} X_k$, in which case we can take $\beta=\gamma=\alpha0^{n_k}\alpha$; or else, we enumerate $X_k=\{\alpha_1,\dots,\alpha_{s_k}\}$ and let $$X_{k+1}=\{\alpha_i 0^{n_k}\alpha_{i+1 \mod s_k}\ |\ i=1,\dots,s_k\}$$
    where we interpret $x \mod s_k \in [1,s_k]$. Thus there exists some $i\in \{1,\dots,s_k\}$ such that $\alpha_i = \alpha$. Take $\beta = \alpha_i 0^{n_k} \alpha_{i+1 \mod s_k}$ and $\gamma = \alpha_{i-1 \mod s_k} 0^{n_k} \alpha_i$.
    \item We may assume that $k$ is the minimum index for which $\alpha$ is a prefix (respectively, suffix) of some word from $X_k$. If $k\leq d$ then the claim follows by (2), so suppose that $k\geq d+1$. 
    Notice that by construction, since $k\geq 2$, $$X_{k} \subseteq \underbrace{X_{k-1}}_{3^{k-2}} \ \underbrace{0^{n_{k-1}}}_{3^{k-2}} \ \underbrace{X_{k-1}}_{3^{k-2}},$$ so if $\alpha$ is a prefix (resp., suffix) of some word from $X_{k}$ and $|\alpha| \leq 3^{d-1} \leq 3^{k-2}$ then $\alpha$ is a prefix (resp., suffix) of a word from $X_{k-1}$, contradicting the minimality of $k$. Hence $k\leq d$ and we are done.
\end{enumerate}
\end{proof}

We next claim that the hereditary closure of $\bigcup_{k=1}^{\infty} X_k$ --- that is, the set of all finite factors of any word from $\bigcup_{k=1}^{\infty} X_k$ --- coincides with the set of all finite factors of some infinite recurrent word $w\in \{0,1,2\}^{\mathbb{N}}$. The proof is standard, but we give it here for completeness.

\begin{prop} \label{prop:infinite word}
There exists an infinite recurrent word $w\in \{0,1,2\}^{\mathbb{N}}$ such that the finite factors of $w$ are precisely the finite factors of $\bigcup_{k=1}^{\infty} X_k$.
\end{prop}

\begin{proof}
Notice that, for every $k\leq k'$, every word from $X_k$ is the prefix of some word from $X_{k'}$ (Lemma \ref{lem:11}). Let $\alpha \in X_i,\beta\in X_j$. Fix some $l\gg 1$ such that $k_l \geq i,j$. Then $\alpha,\beta$ are both prefixes of some $\alpha\alpha',\beta\beta'\in X_{k_l}$. By construction, $X_{k_{l}+1} = X_{k_l} 0^{n_{k_l}} X_{k_l}$, so $\alpha\alpha'0^{n_{k_l}}\beta\beta'\in X_{k_l+1}$, and therefore, there exist some $\gamma,\delta$ such that $\alpha \gamma \beta \delta \in X_{k_l+1}$.

Enumerate (in an arbitrary order) all words in $\bigcup_{k=1}^{\infty} X_k$:
$$u_1,u_2,\dots$$
and notice that by the above argument, there exist some $\gamma_1,\delta_1$ such that $v_2:=u_1\gamma_1 u_2 \delta_1$ belongs to this list. For each $i\geq 2$, suppose that $v_i$ has been defined and appears in the above list. By the above argument, there exist some $\gamma_i,\delta_i$ such that $v_{i+1} := v_i \gamma_i u_{i+1} \delta_i$ belongs to the above list. 
Now $$ w := \lim_{i\rightarrow \infty} v_i = u_1 \ \gamma_1 \ u_2 \ \delta_1 \ \gamma_2 \ u_3 \ \delta_2 \ \gamma_3 \ u_4 \ \delta_3 \ \gamma_4 \ u_5 \ \delta_5 \cdots  $$
is an infinite word as desired. Indeed, each word in $\bigcup_{k=1}^{\infty} X_k$ factors $w$, and every finite factor of $w$ factors some $v_i$, and $v_i\in \bigcup_{k=1}^{\infty} X_k$. Moreover, since every finite factor of any word from $\bigcup_{k=1}^{\infty} X_k$ factors infinitely many members from $\bigcup_{k=1}^{\infty} X_k$, the word $w$ is recurrent.
\end{proof}

\begin{lem} \label{lem:use}
If $\alpha$ is a finite factor of $w$ and $|\alpha|\leq n_d=3^{d-1}$ then $\alpha$ factors some word from $X_{d+1}$.
\end{lem}

\begin{proof}
If $\alpha$ factors $w$ then by Proposition \ref{prop:infinite word}, $\alpha$ factors some word $\beta\in X_p$ for some $p\in \mathbb{N}$. Further assume that $p$ is smallest possible. If $p\leq d+1$, we are done, since every word from $X_i$ factors a word from $X_{i+1}$ for all $i$. Assume toward contradiction that $p\geq d+2$ (hence $p\geq 3$). Now $\alpha$ factors $$\beta\in X_p \subseteq X_{p-1}0^{n_{p-1}} X_{p-1},$$ and since $|\alpha|\leq 3^{d-1}\leq 3^{p-2}=n_{p-1}$, it follows that $\alpha$ factors some word from $X_{p-1}0^{n_{p-1}} \cup 0^{n_{p-1}} X_{p-1}$.

\medskip

\noindent \emph{Case I}: $\alpha$ factors a word from $X_{p-1}0^{n_{p-1}}$, and since $\alpha$ does not factor any word from $X_{p-1}$, it must take the form $\alpha = \alpha' 0 ^i$ for some $\alpha'$ a suffix of a word from $X_{p-1}$ and $0 \leq i\leq |\alpha| \leq 3^{d-1}\leq 3^{p-3} =  n_{p-2}$. Since $$X_{p-1} \subseteq \underbrace{X_{p-2}}_{3^{p-3}} \underbrace{0^{n_{p-2}}}_{3^{p-3}} \underbrace{X_{p-2}}_{3^{p-3}}$$
and $p-3\geq d-1$, it follows that $\alpha'$ --- whose length is at most $|\alpha|\leq 3^{d-1}$ --- is in fact a suffix of some word $\gamma$ from $X_{p-2}$. By Lemma \ref{lem:11}, we know that $\gamma  0^{n_{p-2}}$ is a prefix of a word from $X_{p-1}$, and it follows that $\alpha = \alpha' 0^i$ factors $\gamma 0^{n_{p-2}}$, which is a factor of a word from $X_{p-1}$, contradicting the minimality of $p$.

\medskip

\noindent \emph{Case II}: $\alpha$ factors a word from $0^{n_{p-1}}X_{p-1}$, and thus takes the form $\alpha = 0 ^i \alpha' $ for some $\alpha'$ a prefix of a word from $X_{p-1}$ and $0 \leq i\leq |\alpha| \leq 3^{d-1}\leq 3^{p-3} =  n_{p-2}$. Since $$X_{p-1} \subseteq \underbrace{X_{p-2}}_{3^{p-3}} \underbrace{0^{n_{p-2}}}_{3^{p-3}} \underbrace{X_{p-2}}_{3^{p-3}}$$
and $p-3\geq d-1$, it follows that $\alpha'$ --- whose length is at most $|\alpha|\leq 3^{d-1}$ --- is in fact a prefix of some word $\gamma$ from $X_{p-2}$. By Lemma \ref{lem:11} we know that $0^{n_{p-2}} \gamma$ is a suffix of a word from $X_{p-1}$, and it follows that $\alpha = 0^i \alpha' $ factors $0^{n_{p-2}} \gamma$, which is a subword of a word from $X_{p-1}$, contradicting the minimality of $p$.
\end{proof}

\subsection{Complexity bounds}

We start by showing that $p_w(n)$ is polynomially bounded.

\begin{lem} \label{lem:complexity upper}
We have $p_w(n_k)\leq 4 n_k s_k$ for all $k \geq 1$.
\end{lem}
\begin{proof}
Fix any $k \geq 1$. By Lemma \ref{lem:use}, every factor $\alpha$ of $w$ of length $n_k = 3^{k-1}$ factors some word from $$X_{k+1}\subseteq \underbrace{X_k}_{3^{k-1}} \underbrace{0^{n_k}}_{3^{k-1}} \underbrace{X_k}_{3^{k-1}}$$
and since $|\alpha|=n_k=3^{k-1}$ then $\alpha$ must factor some word from $X_k 0^{n_k} \cup 0^{n_k} X_k$, and hence takes the form $u0^i$ (or $0^iu$) for $u$ a prefix (or a suffix, respectively) of a word from $X_k$, and $0\leq i\leq 3^{k-1}$ (notice that if $\alpha$ factors some word from $X_k$ then $\alpha$ itself is a word from $X_k$ as its length is $n_k$). The number of prefixes (resp., suffixes) of any given length of words from $X_k$ is bounded by $|X_k|=s_k$. Hence the number of words $\alpha$ as above is at most $$\underbrace{2}_{\substack{\text{prefix} \\ \text{or suffix}}}\underbrace{(3^{k-1}+1)}_{\substack{\text{options} \\ \text{for}\ i}} \underbrace{s_k}_{\substack{|X_k| \\ \text{number of} \\ \text{prefixes/suffixes} \\ \text{of length}\ 3^{k-1}-i}}\leq 4n_k s_k,$$
as claimed.
\end{proof}

\begin{cor}
We have $p_w(n) = O ( n^{\alpha 2^r+1} )$.
\end{cor}
\begin{proof}
We first observe that by construction, $s_{k_l}=n_{k_l}^\alpha$ for every $l$. Let $m\geq 2 = k_1$ be arbitrary and let $l\in \mathbb{N}$ be such that $m\in [k_l,k_{l+1}-1]$. By construction, we have $$s_m\leq s_{k_l}^{2^r}\leq n_{k_l}^{\alpha 2^r} \leq n_m^{\alpha 2^r}.$$
Fix any $n \in \mathbb{N}$ and let $m\geq 1$ be such that $n_m \leq n <  n_{m+1}$. By Lemma \ref{lem:complexity upper}, 
\begin{eqnarray*}
p_w(n) & \leq & p_w(n_{m+1}) \\ & \leq & 4 n_{m+1} s_{m+1} \\ & = & 4 n_{m+1} \cdot n_{m+1}^{\alpha 2^r} \\ & = & 4n_{m+1}^{\alpha 2^r+1} \\ & = & 4\cdot 3^{\alpha 2^r + 1} n_m^{\alpha 2^r + 1} \\ & \leq & 4\cdot 3^{\alpha 2^r + 1} n^{\alpha 2^r + 1},
\end{eqnarray*}
as claimed.
\end{proof}

To prove Theorem \ref{thm:Cassaigne2}, we need the following estimates on certain values of the complexity function of $w$. 
Given $l\geq 1$, denote $t_l:=n_{k_l+r}=3^{k_l+r-1}$ (if clear from the context what $l$ is, we omit the subscript of $t_l$).

\begin{lem} \label{lem:p(n+3t)}
There exists a positive constant $K>0$ such that for every $l\in \mathbb{N}$ we have $$ts_{k_{l+1}}^2 + p_w(n_{k_{l+1}}) \leq p_w(n_{k_{l+1}}+3t) \leq K ts_{k_{l+1}}^2.$$
\end{lem}

\begin{proof}
We begin with the upper bound. We first observe that $$n_{k_{l+1}} + 3t = n_{k_{l+1}}+3n_{k_l + r} \leq n_{k_{l+1}} + 3n_{k_{l+1} - 1} = 2n_{k_{l+1}}$$
which is smaller than $n_{k_{l+1}+1}$, so by Lemma \ref{lem:use}, every word $\alpha$ of length $n_{k_{l+1}}+3t$ factors some word from 
\begin{eqnarray*}
X_{k_{l+1}+2} & \subseteq & X_{k_{l+1}+1} 0 ^{n_{k_{l+1}+1}} X_{k_{l+1}+1} \\ & \subseteq & \underbrace{X_{k_{l+1}}}_{A} \underbrace{0^{n_{k_{l+1}}}}_{B} \underbrace{X_{k_{l+1}}}_{C} \underbrace{0^{n_{k_{l+1}+1}}}_{D} \underbrace{X_{k_{l+1}}}_{E} \underbrace{0^{n_{k_{l+1}}}}_{F} \underbrace{X_{k_{l+1}}}_{G}.
\end{eqnarray*}
Since $|\alpha|=n_{k_{l+1}} + 3t \leq 2n_{k_{l+1}} < n_{k_{l+1}+1}$, it follows that $\alpha$ factors some word from
$$ABC \cup BCD \cup DEF \cup EFG.$$
(Notice that $EFG=ABC$ but we included it for clarity.) Therefore $\alpha$ is of one of the following types:

\begin{enumerate}[(I)]
    \item $u0^i$ where $u$ is a suffix of some word from $X_{k_{l+1}}$
    \item $0^i v$ where $v$ is some prefix of some word from $X_{k_{l+1}}$
    \item $0^i u 0^j$ where $u\in X_{k_{l+1}}$
    \item $u0^{n_{k_{l+1}}}v$ where $u$ is a suffix of some word from $X_{k_{l+1}}$ and $v$ is a prefix of some word from $X_{k_{l+1}}$.
\end{enumerate}

We now turn to bound the number of possible words of length $n_{k_{l+1}}+3t$ of each one of types (I)---(IV).

\begin{enumerate}[(I)]
    \item The number of words of length $n_{k_{l+1}}+3t$ taking the form $u0^i$ where $u$ is a suffix of a word from $X_{k_{l+1}}$ is at most $$\sum_{j=0}^{n_{k_{l+1}}} \#\Big\{\substack{\text{Suffixes of words from}\ X_{k_{l+1}} \\ \text{of length}\ j}\Big\} \leq  (n_{k_{l+1}}+1)|X_{k_{l+1}}| \leq 2n_{k_{l+1}} s_{k_{l+1}}.$$
    \item The number of words of length $n_{k_{l+1}}+3t$ taking the form $0^iv$ where $v$ is a prefix of a word from $X_{k_{l+1}}$ is at most $$\sum_{j=0}^{n_{k_{l+1}}} \#\Big\{\substack{\text{Prefixes of words from}\ X_{k_{l+1}} \\ \text{of length}\ j}\Big\} \leq  (n_{k_{l+1}}+1)|X_{k_{l+1}}| \leq 2n_{k_{l+1}} s_{k_{l+1}}.$$
    \item The number of words of length $n_{k_{l+1}}+3t$ taking the form $0^i u 0^j$ where $u\in X_{k_{l+1}}$ is at most 
    $$\sum_{i=0}^{n_{k_{l+1}} + 3t} \#\{\text{Options for }\ u\} = (n_{k_{l+1}}+3t+1) |X_{k_{l+1}}| \leq 3n_{k_{l+1}}s_{k_{l+1}}.$$
    \item The number of words of length $n_{k_{l+1}}+3t$ taking the form $u 0^{n_{k_{l+1}}} v$ where $u$ is a suffix of a word from $X_{k_{l+1}}$ and $v$ is a prefix of a word from $X_{k_{l+1}}$ is at most 
    \begin{eqnarray*} & & \sum_{i=0}^{3t} \#\Big\{\substack{\text{Suffixes of words from}\ X_{k_{l+1}} \\ \text{of length}\ i}\Big\} \cdot \#\Big\{\substack{\text{Prefixes of words from}\ X_{k_{l+1}} \\ \text{of length}\ 3t-i}\Big\}  \\ & & \leq  (3t+1)|X_{k_{l+1}}|^2 \leq   4ts_{k_{l+1}}^2.
\end{eqnarray*}
\end{enumerate}
Altogether, the number of options for $\alpha$ as above is at most:
$$ \underbrace{2n_{k_{l+1}} s_{k_{l+1}}}_{(I)} + \underbrace{2n_{k_{l+1}} s_{k_{l+1}}}_{(II)} + \underbrace{3n_{k_{l+1}}s_{k_{l+1}}}_{(III)}+\underbrace{4ts_{k_{l+1}}^2}_{(IV)} \leq11 t s_{k_{l+1}}^2, $$
where the inequality follows since $s_{k_{l+1}}=n_{k_{l+1}}^{\alpha} \geq n_{k_{l+1}}$.
The upper bound is proved.

\medskip

We now move to the lower bound. Since $w$ is recurrent (Proposition \ref{prop:infinite word}), for each factor $\alpha$ of $w$ of length $n_{k_{l+1}}$, we can consider some extension $u_\alpha \alpha v_\alpha$ where $u_\alpha$ is of length $t$ and $v_\alpha$ is of length $2t$. Clearly, the assignment $\alpha \mapsto u_\alpha \alpha v_\alpha$ is injective. 

Next, for every $u,v \in X_{k_l + r}$, since $k_l + r < k_{l+1}$, Lemma \ref{lem:11} tells us that $u$ is a suffix of some word in $X_{k_{l+1}}$ and $v$ is a prefix of some word in $X_{k_{l+1}}$. Therefore, $u0^{n_{k_{l+1}}} v$ factors a word from $X_{k_{l+1}}0^{n_{k_{l+1}}} X_{k_{l+1}}$, which is equal to $X_{k_{l+1} + 1}$ by construction. For each $u,v\in X_{k_{l}+r}$ and $0\leq i\leq t = n_{k_{l}+r}$, fix an arbitrary extension of $u0^{n_{k_{l+1}}} v$ to a factor of $w$ (again, recall that this is possible as $w$ is recurrent)
$$ \xi_{u,v,i} := \underbrace{\mu}_{i} \underbrace{u}_{t} \underbrace{0^{n_{k_{l+1}}}}_{n_{k_{l+1}}} \underbrace{v}_{t} \underbrace{\nu}_{t-i} $$ for some $\mu$ of length $i$ and $\nu$ of length $t-i$. Recall that $|u|=|v|=t$, and therefore $|\xi_{u,v,i}|=n_{k_{l+1}}+3t$. Moreover, we claim that the assignment 
\begin{eqnarray*}
X_{k_{l}+r} \times X_{k_{l}+r} \times [0,t] & \longrightarrow & \{\text{Subwords of}\ w\ \text{of length}\ n_{k_{l+1}}+3t\} \\
(u,v,i) & \longmapsto & \xi_{u,v,i}
\end{eqnarray*}
is injective. Indeed, $\xi_{u,v,i}$ contains a \emph{unique} occurrence of $0^{n_{k_{l+1}}}$, since $u,v$ start and end with either `$1$' or `$2$' (by Lemma \ref{lem:11}) and $\mu u,v \nu$ are of length $\leq 2t < n_{k_{l+1}}$. This allows us to trace back $u,v$ and therefore $i$. 

Furthermore, we claim that
\begin{eqnarray*}
& & \{u_\alpha \alpha v_\alpha\ |\ \alpha\ \text{factor of}\ w\ \text{of length}\ n_{k_{l+1}}\} \cap \{\xi_{u,v,i}\ |\ u,v\in X_{k_{l}+r},\ 0\leq i\leq t\} \\ & & = \{u_{0\cdots 0}0^{n_{k_{l+1}}} v_{0\cdots 0}\},
\end{eqnarray*}
a singleton.
This can be seen since $u_\alpha,v_\alpha$ end and start with either `$1$' or `$2$' (by Lemma \ref{lem:11}), and since their lengths are $\leq 2t = 2n_{k_{l}+r} < 3n_{k_{l}+r} = n_{k_{l}+r+1} \leq n_{k_{l+1}}$, the only possibility for $u_\alpha \alpha v_\alpha$ to contain an occurrence of $0^{n_{k_{l+1}}}$ is if $\alpha=0^{n_{k_{l+1}}}$.

Therefore:
\begin{eqnarray*}
p_w(n_{k_{l+1}}+3t) & \geq & \Big| \{u_\alpha \alpha v_\alpha\ |\ \alpha\ \text{factor of}\ w\ \text{of length}\ n_{k_{l+1}}\} \cup \\ & &  \{\xi_{u,v,i}\ |\ u,v\in X_{k_{l}+r},\ 0\leq i\leq t\} \Big| \\ & = & | \{u_\alpha \alpha v_\alpha\ |\ \alpha\ \text{factor of}\ w\ \text{of length}\ n_{k_{l+1}}\} | \\ & & + | \{\xi_{u,v,i}\ |\ u,v\in X_{k_{l}+r},\ 0\leq i\leq t\} | - 1 \\ & = & |\{\text{Subwords of}\ w\ \text{of length}\ n_{k_{l+1}}\}| + |X_{k_{l}+r}|^2\cdot (t+1) - 1 \\ & \geq & p_w(n_{k_{l+1}}) + s_{k_{l}+r}^2 t \\ & = & p_w(n_{k_{l+1}}) + s_{k_{l+1}}^2 t
\end{eqnarray*}
(for the last equality, recall that $s_m = s_{k_{l}+r}$ for all $k_l + r + 1 \leq m \leq k_{l+1}$, by construction.) The lemma is proved.
\end{proof}

\begin{proof}[{Proof of Theorem \ref{thm:Cassaigne2}}]
Let $\varepsilon>0$ be given.
We use the word $w$ constructed above; by Proposition \ref{prop:infinite word}, $w$ is recurrent. Recall that the construction depends on the choice of $r$, which we take to be a large enough integer at least $2$ for which $\frac{1}{2^{r-1}} < \varepsilon$. By Lemma \ref{lem:p(n+3t)}, 
\begin{equation*}
\sum_{m=n_{k_{l+1}}+1}^{n_{k_{l+1}}+3t} p_w'(m) = p_w(n_{k_{l+1}}+3t) - p_w(n_{k_{l+1}}) \geq t s_{k_{l+1}}^2
\end{equation*}
so there exists some $m \in [n_{k_{l+1}}+1,n_{k_{l+1}}+3t]$ such that \begin{equation} \label{eq:111} 
p_w'(m)\geq \frac{t s_{k_{l+1}}^2 }{3t} = \frac{1}{3} s_{k_{l+1}}^2. 
\end{equation}
Moreover, again by Lemma \ref{lem:p(n+3t)} (the upper bound),

\begin{equation} \label{eq:222}
\frac{p_w(m)}{m} \leq \frac{p_w(n_{k_{l+1}}+3t)}{n_{k_{l+1}}} \leq \frac{K  t s_{k_{l+1}}^2}{n_{k_{l+1}}}.
\end{equation}
Now $$t=n_{k_{l}+r} = 3^{k_{l}+r-1},\ \ \ n_{k_{l+1}} = 3^{k_{l+1} - 1} = 3^{(k_l - 1)\cdot 2^r}$$
and for $l\gg_r 1$, we have that $k_l \geq r+1$ and so
\begin{equation} \label{eq:333}
n_{k_{l+1}}^{\frac{1}{2^{r-1}}} = 3^{\frac{(k_l - 1)\cdot 2^r}{2^{r-1}}} = 3^{2k_l - 2} \geq 3^{k_l + r -1} = n_{k_{l}+r} = t.
\end{equation}
Combined with Equation (\ref{eq:222}),
$$\frac{p_w(m)}{m} \leq K s_{k_{l+1}}^2 \frac{t}{n_{k_{l+1}}} \leq K s_{k_{l+1}}^2 n_{k_{l+1}}^{\frac{1}{2^{r-1}} - 1}$$
and, combined with Equation (\ref{eq:111}),
\begin{eqnarray*}
\frac{p_w(m)}{m} & \leq & K s_{k_{l+1}}^2 n_{k_{l+1}}^{\frac{1}{2^{r-1}} - 1} \\ & \leq & 3K n_{k_{l+1}}^{\frac{1}{2^{r-1}} - 1} p_w'(m) \\ & \leq & 3K \left(\frac{1}{2}m\right)^{\frac{1}{2^{r-1}}-1}p_w'(m) \\& = & 3K \cdot 2^{1-\frac{1}{2^{r-1}}} \cdot m^{\frac{1}{2^{r-1}}-1}p_w'(m)
\end{eqnarray*}
since $m\leq n_{k_{l+1}}+3t\leq 2n_{k_{l+1}}$ (notice that the power $x = \frac{1}{2^{r-1}-1}-1$ is negative, so $n_{k_{l+1}}^x \leq \left(\frac{1}{2}m\right)^x$).
Since $\frac{1}{2^{r-1}}<\varepsilon$, it follows that for infinitely many $m$'s,
$$\frac{p_w(m)}{m} \leq m^{\varepsilon - 1}p_w'(m)$$
or equivalently,
$$p_w'(m) \geq \frac{p_w(m)}{m^\varepsilon},$$
as claimed.
\end{proof}

\section{Strictly ergodic subshifts}

The following is a known combinatorial description of uniquely ergodic subshifts.

\begin{prop} \label{prop:uniquely}
Let $\mathcal{X}\subseteq \Sigma^{\mathbb{Z}}$ be a subshift. Then $\mathcal{X}$ is uniquely ergodic if and only if for every finite word $u\in \Sigma^*$ the limit (`frequency')
$$
\omega(u):=\lim_{n\rightarrow \infty} \frac{|\text{Occurrences of}\ u\ \text{in}\ w_0\cdots w_{n-1}\}|}{n}
$$
exists uniformly for all $w\in \mathcal{X}$.

Suppose that $\mathcal{X}$ is the shift closure of the infinite word $w\in \Sigma^{\mathbb{Z}}$. Then, equivalently, the above condition becomes -- 
for every finite word $u\in \Sigma^*$ the limit
$$
\omega(u):=\lim_{n\rightarrow \infty} \frac{|\text{Occurrences of}\ u\ \text{in}\ w_i\cdots w_{i+n-1}\}|}{n}
$$
exists uniformly for all $i\in \mathbb{Z}$.
\end{prop}

Our goal in this section is to prove Theorem \ref{thm:strictly}. 
Our technique is inspired by our earlier construction \cite{GMZ} of minimal subshifts with prescribed complexity, which in turn, is a modification of \cite{SB}. However, strict ergodicity is a much stronger condition than minimality and requires a finer analysis.

\subsection{Construction}

Let $\Sigma$ be a set of cardinality $b:=f(1)$. We define a sequence of positive integers $\{c_{k}\ \colon\ k=0,1,2,\dots\}$ and define $N_k:=bc_0c_1\cdots c_k$ accordingly.

\begin{prop} \label{prop:define}
Given a non-decreasing, submultiplicative, subexponential function $f\colon \mathbb{N}\rightarrow \mathbb{N}$, we can define $\{c_k\ :\ k=0,1,2,\dots\}$ as above such that:
\begin{enumerate}
\item $1\leq c_{k+1}\leq N_k$ for all $k\geq 0$ and $c_0\leq b$
\item $f(2^k)\leq N_k \leq 2f(2^{k+1})$ for all $k\geq 0$
\item $c_k=1$ for infinitely many values of $k$.
\end{enumerate}
\end{prop}

\begin{proof}
We put $c_0=1$ and notice that $N_0=bc_0=b$ satisfies $f(2^0)=b = N_0 \leq 2f(2^{0+1})$, where the last inequality follows since $f$ is non-decreasing. Hence the above requirements (1),(2) hold for $k=0$. For each $k\geq 0$ for which $c_{k+1}$ has not been defined yet, we proceed as follows:
\begin{align*} \label{eqn:construction}
\text{If } N_k^2\leq 2f(2^{k+2}): & & & c_{k+1} = N_k \\
\text{If } N_k^2 > 2f(2^{k+2}): & & &c_{k+1}=\lfloor \frac{2f(2^{k+2})}{N_k} \rfloor \nonumber \\ 
& & &c_{k+2} = 1 \nonumber
\end{align*}
and notice that in any case, $1\leq c_{i+1}\leq N_i$; the only non-trivial case is when $i=k$ and $N_k^2 > 2f(2^{k+2})$. Then $$c_{k+1}=\lfloor \frac{2f(2^{k+2})}{N_k}\rfloor \leq \frac{2f(2^{k+2})}{N_k} < N_k$$ and by the induction hypothesis, $N_k\leq 2f(2^{k+2})$ so $c_{k+1}\geq 1$. This proves that the above requirement (1) is fulfilled.

We now turn to examine requirement (2). Suppose that $c_{k+1}$ has not been defined in the previous step and that $N_k^2 \leq 2f(2^{k+2})$. Then $$ N_{k+1}=N_kc_{k+1}=N_k^2\leq 2f(2^{k+2}), $$ and $$ N_{k+1}=N_k^2\geq f(2^k)^2\geq f(2^{k+1}) $$ by submultiplicativity and by the induction hypothesis.
Now suppose that $c_{k+1}$ has not been defined in the previous step and that $N_k^2 > 2f(2^{k+2})$. Then $$ N_{k+1}=N_k \lfloor \frac{2f(2^{k+2})}{N_k} \rfloor\leq 2f(2^{k+2}) $$ and 
\begin{align*}
N_{k+1}=N_k \lfloor \frac{2f(2^{k+2})}{N_k} \rfloor & \geq  N_k\left(\frac{2f(2^{k+2})}{N_k} - 1\right) \\ & \geq   N_k \cdot \frac{f(2^{k+2})}{N_k} = f(2^{k+2}) \geq f(2^{k+1}) \end{align*}
(since $N_k\leq f(2^{k+2})$ and so $\frac{f(2^{k+2})}{N_k} \geq 1$) 
and $N_{k+2}=N_{k+1}$ satisfies $$ N_{k+2}=N_{k+1}\leq 2f(2^{k+2})\leq 2f(2^{k+3}) $$ and $N_{k+2}=N_{k+1}\geq f(2^{k+2})$ as shown above. Therefore, requirement (2) is fulfilled as well.

Finally, we claim that $c_k=1$ infinitely often. We claim that, for infinitely many values of $k$, we have $N_k^2>2f(2^{k+2})$. Otherwise, there exists some $k_0$ such that for all $k\geq k_0$ we have $N_k^2\leq 2f(2^{k+2})$, so $c_{k+1}=N_k$ and thus $N_{k+1}=N_k^2$. It follows that for all $k\geq k_0$ we have that
$$N_k = N_{k_0}^{2^{k-k_0}}$$
so $$f(2^{k+2})\geq \frac{1}{2}N_{k+1}=\frac{1}{2}N_k^2\geq N_k=N_{k_0}^{2^{k-k_0}}$$
and therefore
$$f(2^{k+2})^{\frac{1}{2^{k+2}}}\geq N_{k_0}^{\frac{2^{k-k_0}}{2^{k+2}}}=N_{k_0}^{\frac{1}{2^{k_0+2}}}=:K>1$$
where $K$ is a constant independent of $k$. Hence, taking the limit as $k\rightarrow \infty$, we see that $f$ grows exponentially, contradicting the assumption. Hence $N_k^2 > 2f(2^{k+2})$ infinitely often, which means that $c_{k+2}=1$ infinitely often. This proves that requirement (3) is fulfilled.
\end{proof}

Given a non-decreasing, submultiplicative, subexponential function $f\colon \mathbb{N}\rightarrow \mathbb{N}$, fix a sequence $\{c_k\ \colon\ k=0,1,2,\dots\}$ as in Proposition \ref{prop:define}. Let $\mathcal{S}\subseteq \mathbb{N}$ be the set of indices $k$ for which $c_k=1$, which, as guaranteed, is infinite. 
We define, for each $k=0,1,2,\dots$, sets $$W(k)\subseteq \Sigma^{2^k}, \ C(k)\subseteq W(k),\ U(k)$$ such that $$|C(k)|=c_k, \ |W(k)|=bc_0\cdots c_{k-1}, \  U(k)\subseteq \bigcup_{i=0}^{k} W(i).$$ Moreover, each $U(k)$ is in fact an \emph{ordered list} which can be written as $U(k)=(u^{(k)}_{1},\dots,u^{(k)}_{t_k})$.

We start with $W(0)=\Sigma$ (notice that $|W(0)|=|\Sigma|=b$) and $U(0)$ is a list containing the elements of $W(0)$, arbitrarily ordered. Suppose that $$\{W(i),U(i)\ \colon\ i=0,1,\dots,k\}, \ \{C(i)\ \colon\ i=0,1,\dots,k-1\}$$ have been defined for some $k\geq 0$. 

If $k\notin \mathcal{S}$, or if $k\in \mathcal{S}$ but $2^k<|u^{(k)}_1|$ then we take $C(k)$ to be an arbitrary subset of $W(k)$ of size $c_k$, which is indeed possible since by Proposition \ref{prop:define}, $1 \leq c_k \leq bc_0\cdots c_{k-1}$. We take $W(k+1)=W(k)C(k)$ and update $U(k+1)$ by appending $W(k+1)$, arbitrarily ordered, to the end of $U(k)$.

If $k\in \mathcal{S}$ (so $c_k=1$) and $2^k\leq |u^{(k)}_1|$ then we fix an arbitrary word $v\in W(k)$ whose prefix is $u^{(k)}_1$, and take $C(k)=\{v\}$. We take $$W(k+1)=W(k)C(k)=W(k)\cdot \{v\}$$ and  update $U(k+1)$ by appending $W(k+1)$, arbitrarily ordered, to the end of $U(k)$, and omitting the first word in it (so $u^{(k+1)}_1=u^{(k)}_2$).

Consider the $\mathbb{Z}$-subshift over the alphabet $\Sigma$ consisting of all bi-infinite words all of whose finite subwords factor subwords from $\bigcup_{k=0}^{\infty} W(k)$. Denote this subshift by $\mathcal{X}$. As in \cite{GMZ}, it follows that $\mathcal{X}$ is minimal and that $f(n)\preceq p_\mathcal{X}(n) \preceq nf(n)$. It remains to show that $\mathcal{X}$ is uniquely ergodic.

\subsection{Unique ergodicity}

Give words $u\in \Sigma^d$ and $w\in \Sigma^m$ we let $\Phi_u(w)$ count the occurrences of $u$ as a subword of $w$ and $\varphi_u(w)=\frac{\Phi_u(w)}{m}$. Thus $\varphi_u(w)\in [0,1]$.

\begin{lem} \label{lem:add}
We have $$\Phi_u(w_0)+\Phi_u(w_1)\leq \Phi_u(w_0w_1)\leq \Phi_u(w_0)+\Phi_u(w_1) + |u| - 1.$$
\end{lem}

\begin{proof}
Every occurrence of $u$ in $w_0$ or $w_1$ is in particular an occurrence in $w_0w_1$; and every occurrence of $u$ in $w_0w_1$ either starts and ends in $w_0$ or $w_1$, or starts in $w_0$ but ends in $w_1$, so its starting position must be within the last $|u|-1$ letters of $w_0$.
\end{proof}

Fix a word $u$. Let $I_n$ be the convex closure of $\{ \varphi_u(w)\ \colon\ w\in W(n) \}$. In other words, $I_n=[a_n,b_n]$ where 
\begin{align*} a_n & =  \min\{ \varphi_u(w)\ \colon\ w\in W(n) \},\\ b_n & =  \max\{ \varphi_u(w)\ \colon\ w\in W(n) \}.\end{align*}

\begin{lem} \label{lem:frequencies for W}
Let $u \in \Sigma^d$ be some finite word over $\Sigma$. Then for every $n$, $$I_{n+1} \subseteq I_n + \left[0,\frac{d}{2^{n+1}}\right] $$
and $\l(I_n)\xrightarrow{n\rightarrow \infty} 0$. Moreover, there exists a real number $0\leq L\leq 1$ such that every sequence $\{x_n\}_{n=1}^{\infty}$ such that $x_n\in I_n$ converges to $L$.
\end{lem}

\begin{proof}
Consider any $w\in W(n+1)$. Since $W(n+1)=W(n)C(n)\subseteq W(n)^2$ we can write $w=w_0w_1$ where $w_0,w_1\in W(n)$. Thus $$\varepsilon := \Phi_u(w) - \Phi_u(w_0) - \Phi_u(w_1) \in [0,d]$$ by Lemma \ref{lem:add} and \begin{align*}
\varphi_u(w)=\frac{\Phi_u(w)}{2^{n+1}} & =  \frac{1}{2}\left(\frac{\Phi_u(w_0)}{2^n}+\frac{\Phi_u(w_1)}{2^n} + \frac{\varepsilon}{2^n}\right) \\ & =  \frac{\varphi_u(w_0)+\varphi_u(w_1)}{2} + \frac{\varepsilon}{2^{n+1}}\in I_n + \left[0,\frac{d}{2^{n+1}}\right].
\end{align*}

Suppose that $n\in \mathcal{S}$, so $C(n)=\{v\}$ is a singleton. Then any $w\in W(n+1)$ takes the form $w=w_0v$ for some $w_0\in W(n)$ and 
\begin{align*}
\varphi_u(w)=\frac{\Phi_u(w)}{2^{n+1}} & =  \frac{1}{2}\left(\frac{\Phi_u(w_0)}{2^n}+\frac{\Phi_u(v)}{2^n} + \frac{\varepsilon}{2^n}\right) \\ & =  \frac{\varphi_u(w_0)}{2}+\frac{\varphi_u(v)}{2} + \frac{\varepsilon}{2^{n+1}} \\ & \in   \frac{1}{2}I_n+\left[0,\frac{d}{2^{n+1}}\right] + \frac{\varphi_u(v)}{2}.
\end{align*}
(As before, $\varepsilon = \Phi_u(w) - \Phi_u(w_0)-\Phi_u(v)$.)
Denote $\Delta_n:=\l(I_n)$ and notice that by the above,

\begin{align}
    \Delta_{n+1}\leq \Delta_n+\frac{d}{2^{n+1}}\ \text{for all}\ n \\
    \Delta_{n+1}\leq \frac{1}{2}\Delta_n+\frac{d}{2^{n+1}}\ \text{for all}\ n\in \mathcal{S}.
\end{align}
Fix some $n_0$. We prove by induction that for all $n\geq n_0$, it holds that $$\Delta_{n+1}\leq \frac{1}{2^{k_n}} \Delta_{n_0} + d\left(\frac{1}{2^{n_0+1}}+\cdots+\frac{1}{2^{n+1}}\right)$$
where $k_n=|\mathcal{S} \cap [n_0,n]|$. For $n=n_0$, if $n_0\notin \mathcal{S}$ then we have $k_n=0$ so the claim holds by (1) above; and if $n_0\in \mathcal{S}$ then $k_n=1$ so the claim holds by (2) above. Suppose that the claim holds for some $n\geq n_0$. If $n+1\notin  \mathcal{S}$, then we have by (1) that $$\Delta_{n+2}\leq \Delta_{n+1}+\frac{d}{2^{n+2}}\leq \frac{1}{2^{k_n}}\Delta_{n_0}+d\left(\frac{1}{2^{n_0+1}}+\cdots+\frac{1}{2^{n+1}}\right) + \frac{d}{2^{n+2}}$$ which is the required inequality, as $k_n=k_{n+1}$. If $n+1\in \mathcal{S}$ then by (2) we have that
\begin{align*}
\Delta_{n+2}\leq \frac{1}{2}\Delta_{n+1}+\frac{d}{2^{n+2}} & \leq  \frac{1}{2}\left(\frac{1}{2^{k_n}}\Delta_{n_0}+d\left(\frac{1}{2^{n_0+1}}+\cdots+\frac{1}{2^{n+1}}\right)\right) + \frac{d}{2^{n+2}} \\ & \leq  \frac{1}{2^{{k_n}+1}}\Delta_{n_0}+d\left(\frac{1}{2^{n_0+1}}+\cdots+\frac{1}{2^{n+2}}\right)
\end{align*}
as claimed (notice that $k_{n+1}=k_n+1$ in this case).
We conclude that $$\Delta_{n+1}\leq \frac{1}{2^{k_n}}\Delta_{n_0}+d\left(\frac{1}{2^{n_0+1}}+\cdots+\frac{1}{2^{n+2}}\right) \leq \frac{1}{2^{k_n}}\Delta_{n_0}+\frac{d}{2^{n_0}}$$ for all $n\geq n_0$.
Let $\delta>0$ be given. Fix $n_0$ such that $\frac{d}{2^{n_0}}<\frac{\delta}{2}$. Fix any $N\geq n_0$ large enough such that $k=|\mathcal{S}\cap [n_0,N]|$ satisfies $2^k > \Delta_{n_0}\cdot \frac{2}{\delta}$ (namely, $\frac{1}{2^k}\Delta_{n_0}<\frac{\delta}{2}$). Then by the claim we just proved,
$$\Delta_{N+1}\leq \frac{1}{2^{k_N}}\Delta_{n_0} + \frac{d}{2^{n_0}} < \frac{\delta}{2} + \frac{\delta}{2} = \delta.$$
Therefore, $\Delta_n\xrightarrow{n\rightarrow \infty} 0$.

Finally, fix a sequence $\{x_n\}_{n=1}^{\infty}$ where each $x_n\in I_n$. Write $I_n=[a_n,b_n]$ where $b_n=a_n+\Delta_n$; thus $a_n\leq x_n\leq a_n+\Delta_n$. Since $I_{n+1}\subseteq I_n+\left[0,\frac{d}{2^{n+1}}\right]$, we have that $a_{n+1}\geq a_n$ and the sequence $\{a_n\}_{n=1}^{\infty}$ is bounded and monotone non-decreasing, hence converges to some limit $L$. It follows that $\{a_n+\Delta_n\}_{n=1}^{\infty}$ converges to $L$ as well, since $\lim_{n\rightarrow \infty} \Delta_n = 0$. Hence $\lim_{n\rightarrow \infty} x_n = L$.
\end{proof}

\begin{lem} \label{lem:composition1}
Let $v$ be a (non-trivial) prefix of a word from some $W(t)$. Then $v\in W(m_1)\cdots W(m_s)$ for some $m_1>\cdots>m_s\geq 0$.
\end{lem}
\begin{proof}
By induction on $|v|$. For $|v|=1$ the statement is evident. Write $$|v|=2^{m_1}+\cdots+2^{m_s},\ \ \ m_1>\cdots>m_s\geq 0$$ (binary expansion.) Write $v=v' v''$ where $|v'|=2^{m_1}$. Then $$v'\in W(0)C(0)C(1)\cdots C(m_1-1) =  W(m_1)$$ and $v''$ is a prefix of a word from $C(m_1)C(m_1+1)\cdots C(t-1)$. Since $2^{m_2}+\cdots+2^{m_s} < 2^{m_1}$, we have that $|v''| < 2^{m_1} \leq |v|$ so in fact $v''$ is a prefix of a word from $C(m_1)\subseteq W(m_1)$, and we can apply the induction hypothesis to $v''$. Hence $v''\in W(m_2)\cdots W(m_s)$ and altogether $v=v' v''\in W(m_1)\cdots W(m_s)$.
\end{proof}

\begin{lem} \label{lem:composition2}
Let $v$ be a (non-trivial) suffix of a word from some $W(t)$. Then $v\in W(n_1)\cdots W(n_r)$ for some $0\leq n_1<\cdots<n_r$.
\end{lem}

\begin{proof}
By induction on $|v|$. For $|v|=1$ the statement is evident. Suppose that $v$ is a suffix of a word $z\in W(t)$ where $t$ is minimal. If $|v|\leq 2^{t-1}$ then $v$ is a suffix of a word from $W(t-1)$, as $W(t)=W(t-1)C(t-1)$ and $C(t-1)\subseteq W(t-1)$, contradicting the minimality of $t$. Hence $|v|>2^{t-1}$ so we can write $$|v|=2^{n_1}+\cdots+2^{n_r},\ \ \ 0\leq n_1<\cdots<n_r=t-1$$ 
(if $|v|=2^t$ then $v=z\in W(t)$ and we are done) 
and $v=v' v''$ where $v''\in C(t-1)\subseteq W(t-1)$ and $v'$ is a suffix of a word from $W(t-1)$. By the induction hypothesis applied to $v'$, we know that $v'\in W(n_1)\cdots W(n_{r-1})$ and altogether $$v = v' v'' \in W(n_1)\cdots W(n_{r-1})W(t-1) = W(n_1)\cdots W(n_{r-1})W(n_r),$$
as claimed.
\end{proof}

\begin{lem} \label{lem:composition}
Every finite factor $w$ of $\mathcal{X}$ takes the form $w=u_1\cdots u_rw_1\cdots w_s$ where $u_i\in W(n_i),\ w_j\in W(m_j)$ and $n_1 < \cdots < n_r$ and $m_1 > \cdots  > m_s$.
\end{lem}

\begin{proof}
By induction on the minimum number $t$ such that $w$ factors a word from $W(t)$.
For $t=0$ (and even $t=1$) the statement is evident. For the induction step, suppose $w$ factors $v\in W(t)=W(t-1)C(t-1)$. Breaking $v=v_0v_1$ with $v_0\in W(t-1),\ v_1\in C(t-1)\subseteq W(t-1)$, we may assume that $v_0=v_{00}v_{01},\ v_1=v_{10}v_{11}$ such that $w=v_{01}v_{10}$, for otherwise $w$ factors either $v_0$ or $v_1$, contradicting the minimality of $t$. 
By Lemma \ref{lem:composition1}, we know that $v_{10}\in W(m_1)\cdots W(m_s)$ for some $m_1>\cdots>m_s\geq 0$, and by Lemma \ref{lem:composition2}, we know that $v_{01} \in W(n_1)\cdots W(n_r)$ for some $0\leq n_1<\cdots<n_r$. The lemma is proved.
\end{proof}

\begin{prop} \label{prop:frequencies}
Let $u\in \Sigma^d$ be a finite word. Then there exists some $L\in [0,1]$ such that for every $\varepsilon>0$ there exists some $N_0$ such that for every $N\geq N_0$ and for every factor $\xi$ of $\mathcal{X}$ of length $N$, we have that $|\varphi_u(\xi)-L|<\varepsilon$.
\end{prop}

\begin{proof}
Consider $L$ from Lemma \ref{lem:frequencies for W}. Let $\varepsilon>0$ be given. By Lemma \ref{lem:frequencies for W}, there exists some $t$ such that if $k\geq t$ and $z\in W(k)$ then $|\varphi_u(z)-L|<\varepsilon/3$. Take $N_0\geq t$ large enough such that for every $N\geq N_0$: \begin{align*}
& \frac{2^{t+1}}{N}<\varepsilon/3\ \ \text{and} \\
& \frac{2\log_2 N + 1}{N} d < \varepsilon/3.
\end{align*}

Now consider any factor $\xi$ of $\mathcal{X}$ of length $N\geq N_0$. By Lemma \ref{lem:composition}, we can write $$\xi=u_1\cdots u_r w_1\cdots w_s$$ where $u_i\in W(n_i),\ w_j\in W(m_j)$ and $n_1<\cdots<n_r,\ m_1>\cdots>m_s$. Thus $N=2^{n_1}+\cdots+2^{n_r}+2^{m_1}+\cdots+2^{m_s}$.
Recall that $\varphi_u(\xi)=\frac{\Phi_u(\xi)}{N}$ and by Lemma \ref{lem:add},
\begin{align} \label{eq:add}
\frac{1}{N}\left(\sum_{i=1}^r \Phi_u(u_i)+\sum_{j=1}^s \Phi_u(w_j)\right) & \leq   \varphi_u(\xi) \nonumber \\ & \leq  \frac{1}{N}\left(\sum_{i=1}^r \Phi_u(u_i)+\sum_{j=1}^s \Phi_u(w_j)\right) + \frac{(r+s-1)d}{N}  
\end{align}

Hence
\begin{align*}
\varphi_u(\xi) & \geq  \frac{1}{N}\left(\sum_{\substack{i=1 \\ n_i\geq t}}^r \Phi_u(u_i)+\sum_{\substack{j=1 \\ m_j\geq t}}^s \Phi_u(w_j)\right) \\ & =  \frac{1}{N}\left(\sum_{\substack{i=1 \\ n_i\geq t}}^r 2^{n_i}\varphi_u(u_i)+\sum_{\substack{j=1 \\ m_j\geq t}}^s 2^{m_j}\varphi_u(w_j)\right) \\ & \geq  \frac{1}{N}\left(\sum_{\substack{i=1 \\ n_i\geq t}}^r 2^{n_i}(L-\varepsilon/3)+\sum_{\substack{j=1 \\ m_j\geq t}}^s 2^{m_j}(L-\varepsilon/3)\right) \\ & =  
\frac{L-\varepsilon/3}{N}\left(\sum_{i=1}^r 2^{n_i} - \sum_{\substack{i=1 \\ n_i<t}}^t 2^{n_i} + \sum_{j=1}^s 2^{m_j} - \sum_{\substack{j=1 \\ m_j<t}}^t 2^{m_j}  \right) \\ & = \frac{L-\varepsilon/3}{N}\left(N-\sum_{\substack{i=1 \\ n_i<t}}^t 2^{n_i} - \sum_{\substack{j=1 \\ m_j<t}}^t 2^{m_j}\right) \\ & \geq \frac{L-\varepsilon/3}{N}\left(N-2^t-2^t\right) \\ & = (L-\varepsilon/3)\left(1-\frac{2^{t+1}}{N}\right) \\ & = (L-\varepsilon/3)(1 - \varepsilon/3) =  L -\varepsilon/3- L \varepsilon/3 +\varepsilon^2/9 \geq L - 2\varepsilon/3.
\end{align*}

We now turn to give an upper bound on $\varphi_u(\xi)$. Recall that $n_1,\dots,n_r$ are distinct and $2^{n_1}+\cdots+2^{n_r}\leq N$, and that $m_1,\dots,m_s$ are distinct and $2^{m_1}+\cdots+2^{m_s}\leq N$, and therefore $r,s\leq \log_2 N + 1$.
Moreover, by (\ref{eq:add}),

\begin{align*}
\varphi_u(\xi) & \leq  \frac{1}{N}\left(\sum_{i=1}^r \Phi_u(u_i)+\sum_{j=1}^s \Phi_u(w_j)\right) + \frac{(r+s-1)d}{N}   \\
& \leq  \frac{1}{N}\left(\sum_{i=1}^r \Phi_u(u_i)+\sum_{j=1}^s \Phi_u(w_j)\right) + \frac{(2\log_2 N + 1)d}{N} \\ & \leq  \frac{1}{N}\left(\sum_{i=1}^r \Phi_u(u_i)+\sum_{j=1}^s \Phi_u(w_j)\right) + \varepsilon/3 \\
&  =  \frac{1}{N}\left(\sum_{\substack{i=1 \\ n_i < t}}^r \Phi_u(u_i)+\sum_{\substack{j=1 \\ m_j < t}}^s \Phi_u(w_j)+\sum_{\substack{i=1 \\ n_i \geq t}}^r \Phi_u(u_i)+\sum_{\substack{j=1 \\ m_j\geq t}}^s \Phi_u(w_j)\right) + \varepsilon/3 \\ & \leq  \frac{1}{N}\left(\sum_{\substack{i=1 \\ n_i < t}}^r 2^{n_i}+\sum_{\substack{j=1 \\ m_j < t}}^s 2^{m_j}+\sum_{\substack{i=1 \\ n_i \geq t}}^r 2^{n_i}\varphi_u(u_i)+\sum_{\substack{j=1 \\ m_j\geq t}}^s 2^{m_j}\varphi_u(w_j)\right) + \varepsilon/3 \\ & \leq  \frac{1}{N}\left(2^t+2^t+(L+\varepsilon/3)\sum_{i=1}^r 2^{n_i}+(L+\varepsilon/3)\sum_{j=1}^s 2^{m_j}\right) \\ & \leq  \frac{2^{t+1}}{N}+L+\varepsilon/3 < L + \varepsilon.
\end{align*}

The proposition is proved.
\end{proof}

\begin{proof}[{Proof of Theorem \ref{thm:strictly}}]
By Proposition \ref{prop:frequencies}, every finite word over $\Sigma$ has a well-defined frequency in $\mathcal{X}$. It follows that $\mathcal{X}$ is minimal and uniquely ergodic, and has the desired complexity.
\end{proof}

\section{Simple algebras and the filter dimension}

Let $\mathbbm{k}$ be a field. Let $A$ be a finitely generated (associative, not necessarily commutative) $\mathbbm{k}$-algebra generated by a finite-dimensional subspace $V$ containing $1$. Let $V^n$ be the $\mathbbm{k}$-subspace of $A$ spanned by all products of at most $n$ elements from $V$.
Thus
$$V\subseteq V^2\subseteq \cdots$$
is an exhausting filtration of $A=\bigcup_{n=1}^{\infty} V^n$.
The \emph{growth function} of $A$ with respect to $V$ is
$$\gamma_{A,V}(n) = \dim_{\mathbbm{k}} V^n.$$
This function depends on the choice of $V$, but only up to asymptotic equivalence. Recall that, for non-decreasing functions $f,g\colon \mathbb{N}\rightarrow \mathbb{N}$ we write $f\preceq g$ if $f(n)\leq C g(D n)$ for some $C,D>0$ and for all $n\gg 1$; and $f\sim g$ (asymptotically equivalent) if both $f\preceq g$ and $g\preceq f$.

The \emph{Gel'fand-Kirillov (GK) dimension} of $A$ is defined as
$$\GK(A) = \overline{\lim_{n\rightarrow \infty}} \frac{\log \gamma_{A,V}(n)}{\log n}$$
and is independent of the choice of generating subspace $V$. The GK-dimension represents the optimal degree of polynomial growth of $A$. This invariant has originated from the work of Gel'fand and Kirillov on skew-fields of fractions of universal enveloping algebras of algebraic Lie algebras. 
The GK-dimension is an important invariant in noncommutative algebra and noncommutative algebraic geometry, replacing the role of the classical Krull dimension in commutative algebra; indeed, in the commutative case, the GK-dimension coincides with the classical Krull dimension of $A$. For more on growth of algebras and the GK-dimension, see \cite{KL}.

\medskip

Now assume that $A$ is a simple algebra, that is, $A$ contains no non-zero proper two-sided ideals. This means that, for every $0\neq a\in A$, the ideal $AaA$ contains $1$ and so there exists some $r\in \mathbb{N}$ such that $1\in V^r a V^r$. Bavula defined the return function
$$Ret_{A,V}(n) = \min\{r\in \mathbb{N}\ \colon\ \forall a\in V^n,\ 1\in V^r a V^r\}$$
which is again a well-defined invariant of $A$ when considered up to asymptotic equivalence. The \emph{filter dimension} of $A$ is then defined to be
$$\fdim(A) = \overline{\lim}_{n\rightarrow \infty} \frac{\log Ret_{A,V}(n)}{\log n}$$
and is independent of the choice of $V$.

One of the most fundamental concepts in algebraic analysis is the notion of $D$-modules. In the simplest and most classical case, consider the $n$-th Weyl algebra $\mathcal{A}_n=\mathbb{C}\left<x_1,\dots,x_n,\partial_1,\dots,\partial_n\right>$ of polynomial vector fields on $\mathbb{A}^n$. This is a simple algebra of GK-dimension $2n$, and every non-zero module over it has GK-dimension at least $n$; this follows from Bernstein's Inequality. Modules of the smallest possible GK-dimension $n$ are \emph{holonomic} $D$-modules. The notion of filter dimension enables to establish general Bernstein-type inequalities for more flexible settings of modules over arbitrary finitely generated simple algebras. We refer the reader to \cite{Bav09,Bavula,Bav05,Bav98,Bav00,Bav08}.

\medskip

In the last part of the paper we address:

\begin{ques}[{\cite{Bavula}}]
Which numbers are the possible values of the filter dimension?
\end{ques}

We prove:

\begin{thm} \label{thm:main Bavula}
Let $\mathbbm{k}$ be an arbitrary field. For every real number $\alpha\geq 1$ there exists a finitely generated, $\mathbb{Z}$-graded, simple $\mathbbm{k}$-algebra $A$ whose filter dimension is $\alpha$.
\end{thm}

Our algebras arise as convolution algebras of \'etale groupoids associated with minimal symbolic dynamical systems (see next sections). Such algebras are finitely generated, simple, $\mathbb{Z}$-graded and can be viewed as certain (non-Ore) localizations of monomial algebras. The filter dimensions of such algebras must be at least $1$. (In general, Bavula  \cite{Bavula} showed that $\fdim(A) \geq \frac{1}{2}$ for arbitrary finitely generated simple algebras.)

\begin{prop} \label{prop:prop}
    Let $R=\bigoplus_{i=0}^{\infty} R_i$ be a finitely generated, non-nilpotent, connected graded algebra. If $R\subseteq A = \bigoplus_{i\in \mathbb{Z}} A_i$ is a finitely generated simple algebra then $\fdim(A)\geq 1$.
\end{prop}

\begin{proof}
Let $V=R_0+\cdots+R_d$ be a generating subspace of $R$ and extend it to some generating subspace of $A$, say, $W \subseteq A_{-e}+\cdots+A_e$. Since $R$ is non-nilpotent, $R_n\neq 0$ infinitely often. Fix some $0\neq a_n\in R_n$. Then $a_n\in V^n\subseteq W^n$. If 
\begin{eqnarray*} 
1\in W^ka_n W^k & \subseteq & \left(A_{-ke}+\cdots+A_{ke}\right)A_n\left(A_{-ke}+\cdots+A_{ke}\right) \\ & \subseteq & A_{n-2ke}+\cdots+A_{n+2ke} 
\end{eqnarray*} then $n-2ke\leq 0$ so $k\geq \frac{n}{2e}$. It follows that $Ret_A(n)\geq \frac{n}{2e}$ infinitely often, so $\fdim(A) \geq 1$.
\end{proof}

Therefore, Theorem \ref{thm:main Bavula} characterizes all possible values of the filter dimension of simple $\mathbb{Z}$-graded algebras containing non-nilpotent $\mathbb{N}$-graded subalgebras as described in Proposition \ref{prop:prop}. (Notice that if we drop the subalgebra assumption then \emph{any} simple algebra $A$ becomes graded by putting $A_0=A,\ A_i=0\ \forall i\neq 0$.)

\section{Convolutions algebras}

\subsection{Subshifts and \'etale groupoids}
Convolution algebras of groupoids, also known as Steinberg algebras, have originate from works of Connes, Nekrashevych, Steinberg and others \cite{BCFS,Con82,Nek,Ste10}. The simplicity of convolution algebras is an important and interesting problem which has been extensively studied in various contexts, e.g. \cite{BCFS,CE,Nek,Ste10,Ste16,SS1,SS2}. Here we focus on a special family of convolution algebras associated with the groupoids of the actions $\mathbb{Z}\curvearrowright \mathcal{X}$ for subshifts $\mathcal{X}$. This family has been utilized by Nekrashevych in \cite{Nek} to construct the first examples of finitely generates simple algebras of arbitrary GK-dimensions.

Given a subshift $\mathcal{X}\subseteq \Sigma^{\mathbb{Z}}$, one forms the \'etale groupoid of the action $\mathfrak{G}_\mathcal{X}$ (if $\mathcal{X}=\mathcal{X}_w$ is the subshift generated by a bi-infinite word $w$ we simply denote $\mathfrak{G}_{\mathcal{X}}$ by $\mathfrak{G}_w$). This is a totally disconnected, compact Hausdorff groupoid whose underlying set is $\mathbb{Z} \times \mathcal{X}$ equipped with the product topology, and partial composition is given by
$$(n,T^m(x))\cdot (m,x) = (n+m,x)$$
and the inverse
$$(n,x)^{-1}=(-n,T^n(x)).$$

Fix an arbitrary base field $\mathbbm{k}$, endowed with the discrete topology.
The \emph{convolution algebra} over $\mathbbm{k}$ associated with $\mathcal{X}\subseteq \Sigma^{\mathbb{Z}}$ is
$$\mathbbm{k}[\mathfrak{G}_\mathcal{X}] = \{f\colon \mathfrak{G}_\mathcal{X}\rightarrow \mathbbm{k}\ \colon\ f\ \text{continuous, compactly supported}\}$$
(here $\mathbbm{k}$ is equipped with the discrete topology) where the ring operation is given by convolution
$$(f_1 * f_2)(g) = \sum_{\substack{h_1,h_2\in \mathfrak{G}_\mathcal{X} \\ h_1h_2=g}} f_1(h_1)f_2(h_2).$$
Notice that the characteristic function $\mathbbm{1}_{\{0\}\times \mathcal{X}}$ is the identity of this algebra, and will be occasionally denoted by $1$. 

\subsection{Generators and filtrations}

The convolution algebra is generated by the following characteristic functions:
\begin{eqnarray*}
    \mathbbm{1}_T & \ \  \text{of the set} \ \ & \{1\}\times \mathcal{X} \\
    \mathbbm{1}_{T}^{-1} & \ \  \text{of the set} \ \ & \{-1\}\times \mathcal{X} \\
    \mathbbm{1}_s & \ \  \text{of the set} \ \ & \{(0,x)\ \colon\ x\in \mathcal{X},\ x[0]=s\}
\end{eqnarray*}
for each $s\in \Sigma$. 

\begin{fact}
Given $p,q\in \mathbb{Z}$ and compact open subsets $A,B\subseteq \mathcal{X}$, we have $$\ind_{\{p\}\times A} * \ind_{\{q\}\times B} = \ind_{\{p+q\}\times \left(B\cap T^{-q}(A)\right)}.$$ Indeed, if $(d,x)=(p,a)\cdot (q,b)$ then $d=p+q$ and $a=T^q(b)$ and $x=b$. Hence the support of the convolution on the left hand side is contained in $\{p+q\}\times \left(B\cap T^{-q}(A)\right)$ and conversely, for every $b\in B\cap T^{-q}(A)$ we can write $(p+q,b)=(p,T^q(b))\cdot (q,b)$, uniquely given $p,q$.
In particular,
$\mathbbm{1}_T * \mathbbm{1}_{T}^{-1} = \mathbbm{1}_{T}^{-1} * \mathbbm{1}_T = 1$. Further notice that $\sum_{s\in\Sigma} \mathbbm{1}_s = 1$. Moreover,
$$\ind_{s} * \ind_{T}^{* d} = \ind_{\{0\}\times \{x\in \mathcal{X}\ \colon\ x[0]=s\}} * \ind_{\{d\}\times \mathcal{X}} = \ind_{\{d\}\times \{x\in \mathcal{X}\ \colon\ x[d]=s\}}$$
since 
\begin{eqnarray*}
T^{-d}(\{x\in \mathcal{X}\ \colon\ x[0]=s\}) & = & \{T^{-d}(x)\in \mathcal{X}\ \colon\ x[0]=s\} \\ & = & \{x\in \mathcal{X}\ \colon\ T^d(x)[0]=s\}=\{x\in \mathcal{X}\ \colon\ x[d]=s\}.
\end{eqnarray*}
and
\begin{eqnarray*}
\ind_T^{* (-d)} * \ind_s * \ind_T^{* d} & = & \ind_{\{-d\}\times \mathcal{X}} * \ind_{\{d\}\times \{x\in \mathcal{X}\ \colon\ x[d]=s\}} \\ & = & \ind_{\{0\}\times \{x\in \mathcal{X}\ \colon\ x[d]=s\}}.
\end{eqnarray*}
\end{fact}

\medskip

\noindent Denote
$$ V = \Span_{\mathbbm{k}} \{\ind_T^{\pm 1},\ind_s\ \colon\ s\in \Sigma\}.$$

Every $f\in \mathbbm{k}[\mathfrak{G}_w]$ is continuous and compactly supported into a discrete space, hence assumes a finite number of values and can be written as a (finite) linear combination of characteristic functions of the form $\ind_{\{d\}\times Z}$ where $d\in \mathbb{Z}$ and $Z\subseteq \mathcal{X}$ is a basic compact open set, taking the form $$Z = \bigcap_{i=1}^{n} \mathcal{X}_{m_i,s_i}$$ where $$\mathcal{X}_{m_i,s_i} = \{x\in \mathcal{X}\ \colon\ x[m_i]=s_i\}.$$ Assume that $-N \leq m_1<\cdots<m_n \leq N$ for some $N\geq 1$ (so $n\leq 2N+1$).
Notice that 
\begin{eqnarray*}
    \ind_{\{d\}\times Z} & = & \ind_T^{* d} * \ind_{\{0\}\times \mathcal{X}_{m_1,s_1}} * \cdots * \ind_{\{0\}\times \mathcal{X}_{m_n,s_n}} \\ & = & \ind_T^{* d} * \left(  \ind_T^{* (-m_1)} * \ind_{s_1} * \ind_T^{* m_1} \right) * \left(  \ind_T^{* (-m_2)} * \ind_{s_2} * \ind_T^{* m_2} \right) * \\ & & \cdots * \left(  \ind_T^{* (-m_n)} * \ind_{s_n} * \ind_T^{* m_n} \right) \\ & = & \ind_T^{*(d-m_1)} * \ind_{s_1} * \ind_T^{*(m_1-m_2)} * \ind_{s_2} * \ind_T^{*(m_2-m_3)} * \\ & & \cdots * \ind_T^{*(m_{n-1}-m_n)} * \ind_{s_n} * \ind_T^{* m_n} \\  & \in & V^{|d-m_1|+(m_2-m_1)+\cdots+(m_n-m_{n-1})+|m_n|+n} \\ & = & V^{|d-m_1|+(m_n-m_1)+|m_n|+n} \\ & \subseteq & V^{|d|+|m_1|+2N+|m_n|+n}\subseteq V^{|d|+4N+2N+1}\subseteq V^{|d|+7N}.
\end{eqnarray*}
Let $$W_N = \Span_{\mathbbm{k}} \{ \ind_{\{d\}\times \{x\in \mathcal{X}\ \colon\ x[-N]=s_1,\dots,x[N]=s_{2N+1}\}} \ \colon \ d\in [-N,N],\ s_1,\dots,s_{2N+1}\in \Sigma \}.$$
We thus have $$W_N \subseteq V^{8N}.$$
Conversely, $$V^N \subseteq W_N.$$ Indeed, $V\subseteq W_1$ and $V * W_N + W_N * V \subseteq W_{N+1}$. We also have $W_N * W_M \subseteq W_{N+M}$.

\subsection{Return and recurrence}

Let $w$ be a uniformly recurrent word. For each finite factor $u$ of $w$ let $K_u$ be the minimum integer such that every factor of $w$ of length $K_u$ contains an occurrence of $u$. The \emph{recurrence function} of $w$ is
$$Rec_w(n) = \max\{K_u\ \colon\ u\ \text{a factor of}\ w\ \text{of length}\ n\}$$
or equivalently
$$Rec_w(n) = \min\Big\{K\in \mathbb{N}\ \colon\ \substack{\text{every factor of}\ w\ \text{of length}\ K\ \text{contains} \\ \text{an occurrence of every factor of}\ w\ \text{of length}\ n} \Big\}.$$

\begin{lem} \label{lem:lower bound on Ret}
Let $w\in \Sigma^{\mathbb{Z}}$ be a uniformly recurrent word. Then
$$Rec_w(n) \preceq Ret_{\mathbbm{k}[\mathfrak{G}_w]}(n) + n.$$
\end{lem}
\begin{proof}
Let $w$ be a uniformly recurrent word  and let $\mathcal{X}=\mathcal{X}_w$ be the corresponding subshift.
Fix $n$. Denote $s:=Ret_{\mathbbm{k}[\mathfrak{G}_w], V}(8n)$ and assume to the contrary that $2s + n < Rec_w(n)$. By definition of the recurrence function $Rec$, we know that there exist factors $u=w[d,d+n-1]$ and $v=w[e,e+2s+n-1]$ of $w$, such that $u$ does not factor $v$. 
Consider the element
$$ a = \ind_{\{0\}\times \{x\in \mathcal{X} \ \colon \ x[0,n-1]=u \}} \in W_n \subseteq V^{8n}.$$
By definition of the return function $Ret$, we have that $$\ind_{\{0\}\times \mathcal{X}} \in V^s a V^s \subseteq W_s a W_s$$
and therefore $\ind_{\{0\}\times \mathcal{X}}$ is a linear combination of functions of the form $$(*)\ \ \ \ind_{\{d_i\}\times Z_i} * \underbrace{\ind_{\{0\}\times \{x\in \mathcal{X} \ \colon \ x[0,n-1]=u \}}}_{a} * \ind_{\{d'_i\}\times Z'_i}$$
where $d_i,d'_i\in [-s,s]$. Each one of the above form $(*)$ is itself a linear combination of characteristic functions supported on $\mathbb{Z}\times \{x\in \mathcal{X} \ \colon \ x[p,p+n-1]=u \}$ for $p\in [-s,s]$.
Now $w$ admits a factor $v$ as above of length $2s+n$ which does not contain any occurrences of $u$; equivalently, there exists  some $w'\in \mathcal{X}$ such that $w'[-s,s+n-1]$ does not contain any occurrence of $u$. Therefore each one of the functions of the form $(*)$ vanishes on $(0,w')\in \mathbb{Z} \times \mathcal{X}$, and so does the identity $\ind_{\{0\}\times \mathcal{X}}$, a contradiction. 
It follows that $2Ret_{\mathbbm{k}[\mathfrak{G}_w], V}(8n) + n \geq Rec_w(n)$, and the claim follows.
\end{proof}

\section{Uniformly recurrent words with prescribed polynomial recurrence} \label{sec:construction of words with prescribed recurrence}

\subsection{Construction and basic properties}

We now turn to a specific construction of a uniformly recurrent word $w$, depending on a sequence of positive integers $(n_j)_{j=1}^{\infty}$, $n_j\geq 2 \ \forall j$. Fix the alphabet $\Sigma=\{a,b\}$. Let
\begin{eqnarray*}
&& \alpha_0  = a, \ \ \ \beta_0=b \\
&& \alpha_{j+1}  =  (\alpha_j ^2 \beta_j)^{n_{j+1}}, \ \ \ \beta_{j+1} = (\beta_j ^2 \alpha_j)^{n_{j+1}}.
\end{eqnarray*}

Since $\alpha_{j+1} = \star \alpha_j \star$ for some non-empty words $\star$, it follows that $w:=\lim_{j\rightarrow \infty} \alpha_j$ is a well-defined bi-infinite word. The finite factors of $w$ are, by definition, all finite factors of $\{\alpha_0,\alpha_1,\alpha_2,\dots\}$. 
Let $$N_k := \prod_{j=1}^k (3n_j) = |\alpha_k|=|\beta_k|$$ so $N_k=3n_k N_{k-1}$, and let $$\widetilde{N}_k := 3(n_k-1)N_{k-1} = N_k - 3N_{k-1}$$ for all $k\geq 1$.
Observe that $\frac{1}{2} N_k \leq \widetilde{N}_k\leq N_k$.

\begin{lem} \label{lem:factors alpha beta}
Let $u$ be a factor of $w$ of length $|u|\leq \widetilde{N}_k$. Then $u$ factors either $\alpha_k \beta_k$ or $\beta_k \alpha_k$.
\end{lem}
\begin{proof}
The word $w$ can be written as a bi-infinite word obtained by concatenating $\alpha_k$ and $\beta_k$ (this is true for every $k$). Every factor $u$ of $w$ of length $\widetilde{N}_k\leq N_k$ factors one of $$\alpha_k\alpha_k,\alpha_k\beta_k,\beta_k\alpha_k,\beta_k\beta_k.$$
If $u$ factors $\alpha_k\beta_k$ or $\beta_k\alpha_k$ we are done; suppose that $u$ factors $$\alpha_k\alpha_k=(\alpha_{k-1}\alpha_{k-1}\beta_{k-1})^{2n_k},$$
whose length is $6n_kN_{k-1}$. Then
$$\alpha_k\alpha_k = (\alpha_{k-1}\alpha_{k-1}\beta_{k-1})^i v u v' (\alpha_{k-1}\alpha_{k-1}\beta_{k-1})^j$$
for some $v,v'$ such that $vuv'=(\alpha_{k-1}\alpha_{k-1}\beta_{k-1})^l$. Taking $v,v'$ shortest possible, we see that $$|v|+|v'|\leq |\alpha_{k-1}\alpha_{k-1}\beta_{k-1}|=3N_{k-1}$$ (in fact, $|v|+|v'|$ is either $0$ or $3N_{k-1}$) since $|u|=\widetilde{N}_k$ is a multiple of $3N_{k-1}$, which is the length of $\alpha_{k-1}\alpha_{k-1}\beta_{k-1}$, so $$|vuv'|\leq |u|+3N_{k-1}=3(n_k-1)N_{k-1}+3N_{k-1}=3n_kN_{k-1},$$ hence $l\leq n_k$ and so $u$ factors $vuv' = (\alpha_{k-1}\alpha_{k-1}\beta_{k-1})^{n_k}=\alpha_k$. Similarly, if $u$ factors $\beta_{k-1}\beta_{k-1}$ then $u$ factors $\beta_k$. The lemma is proved.
\end{proof}

\begin{lem} \label{lem:pre every 7}
Every factor of $w$ of length $7N_k$ contains both $\alpha_k \beta_k$ and $\beta_k \alpha_k$ as factors.
\end{lem}

\begin{proof}
First notice that we can decompose $w=\prod_{i\in \mathbb{Z}} w_i$ where $w_i\in \{\alpha_{k+1},\beta_{k+1}\}$. 
Let $u$ be a factor of $w$ of length $7N_k$. Then either $u$ contains $\alpha_{k+1}$ or $\beta_{k+1}$ as factors, in which case we are done as both of them contain $\alpha_k\beta_k$ and $\beta_k\alpha_k$ as factors. Assume the contrary. Then $u$ factors one of $$\alpha_{k+1}\alpha_{k+1},\ \beta_{k+1}\beta_{k+1},\ \alpha_{k+1}\beta_{k+1},\  \beta_{k+1}\alpha_{k+1}.$$

Suppose that $u$ factors $XY$ for some $X,Y\in \{\alpha_{k+1},\beta_{k+1}\}$. Recall that each one of these is a concatenation of $\alpha_k,\beta_k$'s (whose lengths are $N_k$), and since $|u|=7N_k$, it follows that $u$ contains $$u_1u_2u_3u_4u_5 \ \ \text{where}\ \ u_i\in \{\alpha_k,\beta_k\}.$$ 
If $u$ factors $\alpha_{k+1}$ or $\beta_{k+1}$ then $u_1u_2u_3u_4u_5$ contains $\alpha_k\beta_k\alpha_k$ or $\beta_k\alpha_k\beta_k$, respectively, and we are done.

Otherwise, $u=u'u''$ where $u'$ is a suffix of $X$ and $u''$ is a prefix of $Y$. Recall that $|u'|+|u''|=|u|=7N_k$. We deal with the following cases.

\medskip

\emph{Suppose $X=\alpha_{k+1},Y=\alpha_{k+1}$.} If $2N_k\leq |u'|\leq 6N_k$ then $u'$ ends with $\alpha_k\beta_k$ and $|u''|\geq N_k$, so $u''$ starts with $\alpha_k$, and hence $u$ contains $\alpha_k\beta_k\alpha_k$ and we are done. If $|u'|\leq 2N_k$ then $|u''|\geq 5N_k$ so $u''$ starts with $\alpha_k\alpha_k\beta_k\alpha_k\alpha_k$ and we are done again. Finally, if $|u'|\geq 6N_k$ then $u'$ ends with $\alpha_k\alpha_k\beta_k\alpha_k\alpha_k\beta_k$ and we are done again.

\medskip

\emph{Suppose $X=\beta_{k+1},Y=\beta_{k+1}$.} If $2N_k\leq |u'|\leq 6N_k$ then $u'$ ends with $\beta_k\alpha_k$ and $|u''|\geq N_k$, so $u''$ starts with $\beta_k$, and hence $u$ contains $\beta_k\alpha_k\beta_k$ and we are done. If $|u'|\leq 2N_k$ then $|u''|\geq 5N_k$ so $u''$ starts with $\beta_k\beta_k\alpha_k\beta_k\beta_k$ and we are done again. Finally, if $|u'|\geq 6N_k$ then $u'$ ends with $\beta_k\beta_k\alpha_k\beta_k\beta_k\alpha_k$ and we are done again.

\medskip

\emph{Suppose $X=\alpha_{k+1},Y=\beta_{k+1}$.} Either $|u'|\geq 2N_k$, in which case $u'$ ends with $\alpha_k\beta_k$, or $|u''|\geq 5N_k$, in which case $u''$ starts with $\beta_k\beta_k\alpha_k\beta_k\beta_k$, so again $u$ contains $\alpha_k\beta_k$. Similarly, either $|u'|\geq 4N_k$, in which case $u'$ ends with $\beta_k\alpha_k\alpha_k\beta_k$, or $|u''|\geq 3N_k$, in which case $u''$ starts with $\beta_k\beta_k\alpha_k$, and in either case $u$ contains $\beta_k\alpha_k$

\medskip

\emph{Suppose $X=\beta_{k+1},Y=\alpha_{k+1}$.} Either $|u'|\geq 2N_k$, in which case $u'$ ends with $\beta_k\alpha_k$, or $|u''|\geq 5N_k$, in which case $u''$ starts with $\alpha_k\alpha_k\beta_k\alpha_k\alpha_k$, so again $u$ contains $\beta_k\alpha_k$. Similarly, either $|u'|\geq 4N_k$, in which case $u'$ ends with $\alpha_k\beta_k\beta_k\alpha_k$, or $|u''|\geq 3N_k$, in which case $u''$ starts with $\alpha_k\alpha_k\beta_k$, and in either case $u$ contains $\alpha_k\beta_k$.
\end{proof}

\subsection{Complexity}

\begin{prop} \label{prop:linear complexity}
Let $w$ be as above. Then $p_w(n)=O(n)$.
\end{prop}

\begin{proof}
Let $N_k \leq n\leq \widetilde{N}_{k+1}$. Let $u$ be a length-$n$ factor of $w$. By Lemma \ref{lem:factors alpha beta}, $u$ factors $\alpha_{k+1}\beta_{k+1}$ or $\beta_{k+1}\alpha_{k+1}$. Let us count the length-$n$ factors of these words.

Let $$V:=\alpha_{k+1}\beta_{k+1} = \underbrace{\alpha_k\alpha_k\beta_k\cdots \alpha_k\alpha_k\beta_k}_{n_k\ \text{times}}\underbrace{\beta_k\beta_k\alpha_k\cdots \beta_k\beta_k\alpha_k}_{n_k\ \text{times}},$$
written as $V=V[1]\cdots V[2N_{k+1}]$. Thus $u=V[i,i+n-1]$ for some $1\leq i\leq 2N_{k+1}-n+1$.
If $1\leq i,i'\leq N_{k+1} - n + 1 $ satisfy $i\equiv i' \mod 3N_k$ then $V[i,i+n-1]=V[i',i'+n-1]$. Likewise, if $N_{k+1}+1 \leq i,i'\leq 2N_{k+1}-n+1$ satisfy $i\equiv i' \mod 3N_k$ then $V[i,i+n-1]=V[i',i'+n-1]$. It follows that $V$ has at most $3N_k$ factors of length $n$ ending before $V[N_{k+1}]$, and at most $3N_k$ factors of length $n$ starting at $V[n_{k+1}+1]$ or afterwards. Any other factor must be of the form $V[i,i+n-1]$ for $N_{k+1}-n+1<i\leq N_{k+1}$, and there are at most $n$ distinct factors of this form. It follows that the number of length-$n$ factors of $\alpha_{k+1}\beta_{k+1}$ is at most $6N_k+n\leq 7n$, and a similar analysis shows that the number of length-$n$ factors of $\beta_{k+1}\alpha_{k+1}$ is at most $7n$ as well. Hence $p_w(n)\leq 14n$.

Now for an arbitrary $n\geq \widetilde{N}_1$, let $k$ be such that $\widetilde{N}_k\leq n < \widetilde{N}_{k+1}$. If $n\geq N_k$ then the above computation shows that $p_w(n)\leq 14n$. If $\widetilde{N}_k \leq n < N_k$ then $$p_w(n)\leq p_w(N_k)\leq 14 N_k \leq 28 \widetilde{N}_k \leq 28 n.$$ It follows that $p_w(n)=O(n)$.
\end{proof}

\subsection{Density of letters within factors}

Given a word $u\in \Sigma^m$ and a letter $x\in \Sigma$ we let $\Phi_x(u)$ count the occurrences of $x$ as a subword of $u$ and $\varphi_x(u)=\frac{\Phi_x(u)}{m}$. Thus $\varphi_x(u)\in [0,1]$. Notice that $\wt_x(u^i)=\wt_x(u)$ for every $i\geq 1$.

\begin{lem} \label{lem:density}
We have $$\wt_a(\alpha_k)=\frac{1}{2}(1+3^{-k}),\ \ \wt_b(\alpha_k)=\frac{1}{2}(1-3^{-k})$$ and $$\wt_a(\beta_k)=\frac{1}{2}(1-3^{-k}),\ \ \wt_b(\beta_k)=\frac{1}{2}(1+3^{-k}).$$
\end{lem}
\begin{proof}
We have $$\Phi_a(\alpha_{k+1})=\Phi_a((\alpha_k^2\beta_k)^{n_{k+1}})=2n_{k+1}\Phi_a(\alpha_k)+n_{k+1}\Phi_a(\beta_k)$$ so \begin{eqnarray*}
\wt_a(\alpha_{k+1}) & = & \frac{\Phi_a(\alpha_{k+1})}{N_{k+1}} \\ & = & \frac{2n_{k+1}\Phi_a(\alpha_k)+n_{k+1}\Phi_a(\beta_k)}{3n_{k+1}N_k} \\ & = & \frac{2}{3}\wt_a(\alpha_k)+\frac{1}{3}\wt_a(\beta_k)
\end{eqnarray*} and likewise $$\wt_a(\beta_{k+1})=\frac{1}{3}\wt_a(\alpha_k)+\frac{2}{3}\wt_a(\beta_k).$$ Hence
$$
\left(\begin{matrix} \wt_a(\alpha_{k+1}) \\ \wt_a(\beta_{k+1})  \end{matrix}\right) = \left(\begin{matrix} \frac{2}{3} & \frac{1}{3} \\ \frac{1}{3} & \frac{2}{3} \end{matrix}\right) \left(\begin{matrix} \wt_a(\alpha_k) \\ \wt_a(\beta_k)  \end{matrix}\right)
$$
and since $\wt_a(\alpha_0)=1,\wt_a(\beta_0)=0$ we have that
\begin{eqnarray*}
\left(\begin{matrix} \wt_a(\alpha_k) \\ \wt_a(\beta_k)  \end{matrix}\right) & = & \left(\begin{matrix} \frac{2}{3} & \frac{1}{3} \\ \frac{1}{3} & \frac{2}{3} \end{matrix}\right)^k \left(\begin{matrix} 1 \\ 0  \end{matrix}\right) \\ & = & \left(\begin{matrix} 1 & 1 \\ 1 & -1 \end{matrix}\right)\left(\begin{matrix} 1 & 0 \\ 0 & \frac{1}{3} \end{matrix}\right)^k\left(\begin{matrix} \frac{1}{2} & \frac{1}{2} \\ \frac{1}{2} & -\frac{1}{2} \end{matrix}\right) \left(\begin{matrix} 1 \\ 0  \end{matrix}\right) \\ & = & \left(\begin{matrix} 1 & 1 \\ 1 & -1 \end{matrix}\right)\left(\begin{matrix} 1 & 0 \\ 0 & 3^{-k} \end{matrix}\right) \left(\begin{matrix} \frac{1}{2} \\ \frac{1}{2}  \end{matrix}\right) \\ & = & \left(\begin{matrix} 1 & 1 \\ 1 & -1 \end{matrix}\right)  \left(\begin{matrix} \frac{1}{2} \\ \frac{3^{-k}}{2}  \end{matrix}\right) =  \left(\begin{matrix} \frac{1}{2}(1+3^{-k}) \\ \frac{1}{2}(1-3^{-k})  \end{matrix}\right).
\end{eqnarray*}
The result for $\wt_b(\alpha_k),\wt_b(\beta_k)$ is analogous.
\end{proof}

\subsection{Recurrence and aperiodicity of factors}

\begin{lem} \label{lem:occurrence beta^3}
\begin{enumerate}
\item Let $V:=\alpha_{k+1}\beta_{k+1}$, written as $V=V[1]\cdots V[2N_{k+1}]$. Suppose that $i$ is such that $V[i,i+3N_k-1]=\beta_k^3$. Then $N_{k+1} - 3N_k + 2 \leq i \leq N_{k+1}$. 

\item Let $V:=\beta_{k+1}\alpha_{k+1}$, written as $V=V[1]\cdots V[2N_{k+1}]$. Suppose that $i$ is such that $V[i,i+3N_k-1]=\alpha_k^3$. Then $N_{k+1} - 3N_k + 2 \leq i \leq N_{k+1}$.
\end{enumerate}
\end{lem}

\begin{proof}
It suffices to prove (1) as (2) is perfectly analogous.
Recall
$$V = \alpha_{k+1}\beta_{k+1}=\underbrace{\alpha_k\alpha_k\beta_k\cdots\alpha_k\alpha_k\beta_k}_{V[1,N_{k+1}]} \underbrace{ \beta_k\beta_k\alpha_k\cdots \beta_k\beta_k\alpha_k}_{V[N_{k+1}+1,2N_{k+1}]}.$$
Suppose that $i\leq N_{k+1}-3N_k+1$. Then $V[i,i+3N_k-1]$ is a factor of length $3N_k$ of $$\alpha_{k+1} = \alpha_k\alpha_k\beta_k\cdots\alpha_k\alpha_k\beta_k$$
and therefore equals, up to a cyclic permutation of its letters, to $\alpha_k\alpha_k\beta_k$, and hence
$$\wt_b(V[i,i+3N_k-1])=\wt_b(\alpha_k\alpha_k\beta_k)=\wt_b(\alpha_{k+1})=\frac{1}{2}(1-3^{-(k+1)})$$
by Lemma \ref{lem:density}. 

Now suppose that $i\geq N_{k+1}+1$; then $V[i,i+3N_k-1]$ is a factor of length $3N_k$ of $$\beta_{k+1} = \beta_k\beta_k\alpha_k\cdots\beta_k\beta_k\alpha_k$$
and therefore equals, up to a cyclic permutation of its letters, to $\beta_k\beta_k\alpha_k$, and hence
$$\wt_b(V[i,i+3N_k-1])=\wt_b(\beta_k\beta_k\alpha_k)=\wt_b(\beta_{k+1})=\frac{1}{2}(1+3^{-(k+1)})$$
again by Lemma \ref{lem:density}.

In either case, 
\begin{eqnarray*} 
\wt_b(V[i,i+3N_k-1]) & \leq & \frac{1}{2}(1+3^{-(k+1)}) \\ & < & \frac{1}{2}(1+3^{-k}) =  \wt_b(\beta_k)=\wt_b(\beta_k^3),\end{eqnarray*}
so $V[i,i+3N_k-1]\neq \beta_k^3$.
\end{proof}

\begin{lem} \label{lem:aperiodic}
Let $k\geq 1$. Let $0<d \leq \widetilde{N}_k = N_k-3N_{k-1}$. Then it is impossible that $$\alpha_k\beta_k[i]=\alpha_k \beta_k[i+d]\ \ \text{for all}\ 1\leq i\leq 2N_k-d,$$ and it is impossible that $$\beta_k\alpha_k[i]=\beta_k\alpha_k[i+d]\ \ \text{for all}\ 1\leq i\leq 2N_k-d.$$
\end{lem}
\begin{proof}
Let $V=\alpha_k\beta_k$. Suppose that $0<d\leq N_k - 3N_{k-1}$ is such that $V[i]=V[i+d]$ for all $1\leq i\leq 2N_k-d$. Recall that $$V = \alpha_k\beta_k=\underbrace{\alpha_{k-1}\alpha_{k-1}\beta_{k-1}\cdots\alpha_{k-1}\alpha_{k-1}\beta_{k-1}}_{V[1,N_k]} \underbrace{ \beta_{k-1}\beta_{k-1}\alpha_{k-1}\cdots \beta_{k-1}\beta_{k-1}\alpha_{k-1}}_{V[N_k+1,2N_k]}.$$ Notice that $V[N_k-N_{k-1}+1,N_k+2N_{k-1}]=\beta_{k-1}^3$. Let $m:=\lceil \frac{N_{k-1}}{d} \rceil$ and observe that
\begin{eqnarray*}
N_k-N_{k-1}+1+md & \geq & N_k-N_{k-1}+1+\frac{N_{k-1}}{d} d \\ & = & N_k+1
\end{eqnarray*}
and
\begin{eqnarray*}
N_k+2N_{k-1}+md & \leq & N_k+2N_{k-1}+\left(\frac{N_{k-1}}{d} + 1\right)d \\ & \leq & N_k+3N_{k-1}+d \\ & \leq & 2N_k.
\end{eqnarray*}
Hence
\begin{eqnarray*}
\beta_{k-1}^3 & = & V[N_k-N_{k-1}+1,N_k+2N_{k-1}] \\ & = & V[N_k-N_{k-1}+1+md,N_k+2N_{k-1}+md]
\end{eqnarray*}
is a factor of $V[N_k+1,2N_k]$, contradicting Lemma \ref{lem:occurrence beta^3}. The proof for $V=\beta_k\alpha_k$ is completely analogous.
\end{proof}

We now further specify the sequence $(n_j)_{j=1}^{\infty}$. Given a real number $\gamma > 1$, we take $n_1=2$ and $$n_{j+1}=\max\left\{2,\lceil \frac{N_j^{\gamma-1}}{3}\rceil\right\}.$$ Thus there exists some $j_0$ such that $n_j=\lceil\frac{N_{j-1}^{\gamma-1}}{3}\rceil$ for all $j\geq j_0$. It follows that

\begin{align} \label{ineq:N_{j+1}}
    N_j^\gamma = N_j^{\gamma-1} \cdot N_j \leq \underbrace{3n_{j+1}N_j}_{=N_{j+1}} \leq 3\left(\frac{N_j^{\gamma-1}}{3} + 1 \right) \cdot N_j \leq 2N_j^{\gamma}
\end{align}
for all $j\geq j_1$ for some constant $j_1$.

\begin{cor} \label{cor:lower bound on Rec}
We have that $Rec_w(n) \geq 3^{-\gamma}n^\gamma$ infinitely often.
\end{cor}
\begin{proof}
By Lemma \ref{lem:occurrence beta^3}, the factor $\alpha_{k+1}$ does not contain any occurrence of $\beta_k^3$. Hence $$Rec_w(3N_k) = Rec_w(|\beta_k^3|)\geq |\alpha_{k+1}|=N_{k+1}$$
and by (\ref{ineq:N_{j+1}}), for every $k \gg 1$,
$$Rec_w(3N_k) \geq N_{k+1} \geq N_k^\gamma = 3^{-\gamma}(3N_k)^\gamma,$$
proving that $Rec_w(n) \geq 3^{-\gamma}n^\gamma$ infinitely often.
\end{proof}

\begin{proof}[{Proof of Proposition \ref{prop:arbitrary rec}}]
By Corollary \ref{cor:lower bound on Rec}, we see that $$\overline{\lim_{n\rightarrow \infty}} \log_n Rec_w(n) \geq \gamma.$$
By Lemma \ref{lem:factors alpha beta}, every factor $u$ of $W$ of length $|u|\leq \widetilde{N}_k$ factors $\alpha_k\beta_k$ or $\beta_k\alpha_k$, and by Lemma \ref{lem:pre every 7}, every factor of $w$ of length $7N_k$ contains both $\alpha_k \beta_k$ and $\beta_k \alpha_k$ as factors. Hence $Rec_w(\widetilde{N}_k)\leq 7N_k$. 

Let $n \geq \widetilde{N}_1$ be arbitrary. Let $k$ be such that $\widetilde{N}_k \leq n\leq \widetilde{N}_{k+1}$. Then (recall (\ref{ineq:N_{j+1}})), $$Rec_w(n)\leq Rec_w(\widetilde{N}_{k+1})\leq 7N_{k+1} \leq 14 N_k^\gamma\leq 14\cdot 2^\gamma \cdot \widetilde{N}_k^\gamma\leq cn^\gamma$$ for some constant $c>0$. Hence, together with the above observation, $$\overline{\lim_{n\rightarrow \infty}} \log_n Rec_w(n)=\gamma.$$ 
By Proposition \ref{prop:linear complexity}, the complexity of $w$ is linear.
\end{proof}

\section{Convolution algebras of arbitrary filter dimensions}

We next apply the words constructed in Proposition \ref{prop:arbitrary rec} to construct finitely generated simple algebras of arbitrary filter dimensions, relating the recurrence functions of these words with the return functions of the associated convolution algebras; throughout, we establish a quantitative form of the classical simplicity argument for convolution algebras of minimal, essentially principal Hausdorff \'etale groupoids (see \cite[Theorem 4.1]{BCFS} and \cite[Proposition 4.1]{Nek} as well as \cite{Ste16,CE}).

\begin{prop} \label{prop:upper bound Ret}
Retain the construction and definitions from Section \ref{sec:construction of words with prescribed recurrence}. Then $$Ret_{\mathbbm{k}[\mathfrak{G}_w]}(n) \preceq n^\gamma.$$
\end{prop}

\begin{proof}
Fix $n \gg 1$ and $0\neq f\in W_n$, written as $f=\sum_{i=1}^m \alpha_i \ind_{F_i}$ where $F_i=\{d_i\}\times Z_i$, each $d_i\in [-n,n]$ and $$Z_i = \{x\in \mathcal{X}\ \colon \ x[-n]=z_{i,-n},\dots,x[n]=z_{i,n} \}$$
for some $z_{i,j}\in \Sigma$. Since $f\neq 0$, it follows that for some $g=(k,\xi)\in  \mathbb{Z}\times \mathcal{X}$, we have that $f(g)\neq 0$. Denote $A = \{1\leq i\leq m\ \colon \ g\in F_i\}$. Thus
$$0\neq f(g)=\sum_{i=1}^m \alpha_i \ind_{F_i}(g) = \sum_{ i \in A } \alpha_i.$$
Consider any compact open subset $\emptyset \neq W\subseteq \mathcal{X}$ such that:
\begin{enumerate}[(i)]
    \item $\{k\} \times W \subseteq \bigcap_{i\in A} F_i$
    \item $\left(\{k\}\times W\right)\cap F_j = \emptyset$ for all $j\notin A$
    \item $W\cap T^d(W) = \emptyset$ for all $0<|d|\leq 2n$.
\end{enumerate} For instance, if we take $W \subseteq \{x\in \mathcal{X}\ \colon\ x[-n,n]=\xi[-n,n] \}$ then $\{k\} \times W\subseteq F_i$ for all $i\in A$ and if $j\notin A$ then $\left(\{k\} \times W\right) \cap F_j = \emptyset$.
Given such $W$, let us compute
\begin{eqnarray} \label{eq:1f1}
\ind_{\{-k\} \times T^k(W)} * f * \ind_{\{0\}\times W} & = & \sum_{i\in A} \ind_{\{-k\} \times T^k(W)} * \alpha_i \ind_{\{d_i\} \times Z_i} * \ind_{\{0\}\times W} \\ & + & \sum_{j\notin A} \ind_{\{-k\} \times T^k(W)} * \alpha_j \ind_{\{d_j\} \times Z_j} * \ind_{\{0\}\times W}. \nonumber
\end{eqnarray}
If $i\in A$, then $d_i = k$ and $W \subseteq Z_i$, hence
\begin{eqnarray*}
\ind_{\{-k\} \times T^k(W)} * \ind_{\{d_i\} \times Z_i} * \ind_{\{0\} \times W} & = & \ind_{\{-k\} \times T^k(W)} * \ind_{\{d_i\} \times W\cap Z_i} \\ & = & \ind_{\{-k\} \times T^k(W)} * \ind_{\{k\}\times W} \\ & = & \ind_{\{0\} \times W}.
\end{eqnarray*}

Now let us focus on the second sum in (\ref{eq:1f1}). If $j\notin A$ then either $Z_j \cap W = \emptyset$ or $d_j \neq k$. 
If $Z_j\cap W = \emptyset$ then
\begin{eqnarray*} \ind_{\{-k\} \times T^k(W)} * \ind_{\{d_j\} \times Z_j} * \ind_{\{0\}\times W} & = & \ind_{\{-k\} \times T^k(W)} * \ind_{\{d_j\} \times W \cap Z_j } \\ & = & \ind_{\{-k\} \times T^k(W)} * \ind_{\emptyset} = 0. \end{eqnarray*}
If $d_j \neq k$ then
 \begin{eqnarray*}
 \ind_{\{-k\} \times T^k(W)} * \ind_{\{d_j\} \times Z_j} * \ind_{\{0\}\times W} & = & \ind_{\{d_j-k\} \times Z_j \cap T^{k-d_j}(W)} * \ind_{\{0\}\times W} \\ & = & \ind_{\{d_j-k\} \times Z_j \cap T^{k-d_j}(W) \cap W} \\ & = & \ind_\emptyset = 0
 \end{eqnarray*}
since $k-d_j\neq 0 $ and $|k-d_j|\leq |k|+|d_j|\leq  2n$ since $d_j \in [-n,n]$ and $k$ must be equal to one of the $d_i$'s --- for otherwise $f(g)=f(k,\xi)=0$ --- and in particular, $|k|\leq n$ as well. For conclusion,
\begin{eqnarray} \label{eq:1f1 2}
\ind_{\{-k\} \times T^k(W)} * f * \ind_{\{0\}\times W} & = & \left( \sum_{i \in A} \alpha_i \right) \ind_{\{0\}\times W}.
\end{eqnarray}

We now estimate `how complicated' $W$ should be. Let $l$ be the minimum index such that $2n + 1 \leq \widetilde{N}_{l+1}$ (hence $2n + 1 > \widetilde{N}_l$). Recall that we assume $n\gg 1$ so we may assume $l\gg 1$ is large enough to ensure the bounds (\ref{ineq:N_{j+1}}). Then by Lemma \ref{lem:factors alpha beta}, the word $\xi[-n,n]$ factors either $\alpha_{l+1} \beta_{l+1}$ or $\beta_{l+1} \alpha_{l+1}$, in either case, of length 
\begin{align} \label{eq:ineq N}
2N_{l+1} \leq 4 N_l^\gamma \leq 2^\gamma 4 \widetilde{N}_{l}^\gamma < 2^\gamma 4 (2n+1)^\gamma\leq Kn^\gamma
\end{align}
for a suitable constant $K\geq 1$. 

Without loss of generality, assume that $\xi[-n,n]$ factors $\alpha_{l+1} \beta_{l+1}$, say, $\alpha_{l+1} \beta_{l+1} = \phi \xi[-n,n] \psi$ for some $|\phi|=p,|\psi|=q$. We let $$W=\{x\in \mathcal{X}\ \colon\ x[-n-p,n+q] = \alpha_{l+1} \beta_{l+1} \}.$$ Notice that $W \subseteq \{x\in \mathcal{X}\ \colon\ x[-n,n]=\xi[-n,n]\}$. Hence conditions (i),(ii) above are satisfied. Regarding condition (iii), suppose that $W \cap T^d(W)\neq \emptyset$ for some $d\neq 0$, $-2n\leq d\leq 2n$. Then there exists some $x\in W$ --- thus $x[-n-p,n+q]=\alpha_{l+1} \beta_{l+1}$ --- such that $x\in T^d(W)$, implying $x[-n-p-d,n+q-d]=\alpha_{l+1} \beta_{l+1}$. It follows that for every $1\leq i\leq 2N_{l+1}-|d|$, we have that $$\alpha_{l+1} \beta_{l+1}[i]=\alpha_{l+1}\beta_{l+1}[i+|d|],$$ contradicting Lemma \ref{lem:aperiodic} (recall that $|d|\leq 2n \leq \widetilde{N}_{l+1})$. 
Hence $$\ind_{\{0\}\times W} \in W_{2N_{l+1}}.$$ 
Together with (\ref{eq:1f1 2}), we have
$$
\underbrace{\ind_{\{-k\} \times T^k(W)}}_{\in W_{2N_{l+1} + n}} * f * \underbrace{\ind_{\{0\} \times W}}_{\in W_{2N_{l+1}}} = \underbrace{\left(\sum_{i\in A} \alpha_i\right)}_{\neq 0} \ind_{\{0\} \times W}
$$
so
$$ \ind_{\{0\}\times W} \in W_{2N_{l+1} + n} * f * W_{2N_{l+1}}. $$
We now need the following:
\begin{lem} \label{lem:every 8}
There exists some (universal) constant $c\in \mathbb{N}$ such that $$\ind_{\{0\}\times \mathcal{X}} \in W_{cN_{l+1}} * \ind_{\{0\}\times W} * W_{cN_{l+1}}.$$
\end{lem}
\begin{proof}
Given $u\in \Sigma^*$, we let $I_u:=\{x\in \mathcal{X}\ \colon\ x[0,|u|-1]=u\}$.
We can write $$\ind_{\{0\}\times \mathcal{X}} = \sum_{u\in L_w(7N_{l+1})} \ind_{\{0\}\times I_u}.$$ By Lemma \ref{lem:pre every 7}, for each $u\in L_w(7N_{l+1})$ we can write $u=\eta_u \alpha_{l+1} \beta_{l+1} \theta_u$. Denote $|\eta_u|=e_u,\ |\theta_u|=t_u$ so $e_u+t_u=5N_{l+1}$.
Now:
\begin{eqnarray*}
    \ind_{\{0\}\times I_u} & = & \ind_{\{0\}\times \{x\in \mathcal{X}\ \colon\ x[0,e_u-1]=\eta_u\}} * \\ & & \ind_{\{0\}\times \{x\in \mathcal{X} \ \colon \ x[e_u,e_u+2N_{l+1}-1]=\alpha_{l+1}\beta_{l+1}\}} * \\ & & \ind_{\{0\}\times \{x\in \mathcal{X} \ \colon \ x[e_u+2N_{l+1},7N_{l+1}-1]=\theta_u\}} \\
    & = & \ind_{\{0\}\times I_{\eta_u}} * \left( \ind_{\{-(n+p+e_u)\}\times \mathcal{X}} * \ind_{\{0\}\times W} * \ind_{\{n+p+e_u\}\times \mathcal{X}} \right) \\ & & * \left( \ind_{\{-(e_u + 2N_{l+1})\}\times \mathcal{X}} * \ind_{\{0\}\times I_{\theta_u}} * \ind_{\{e_u + 2N_{l+1}\}\times \mathcal{X}} \right) \\
    & = & \underbrace{\ind_{\{0\}\times I_{\eta_u}}}_{\in W_{e_u}} * \underbrace{\ind_{\{-(n+p+e_u)\}\times \mathcal{X}}}_{\in W_{n+p+e_u}} * \ind_{\{0\}\times W} * \\ & & \underbrace{\ind_{\{-(n+q+1)\}\times \mathcal{X}}}_{\in W_{n+q+1}} * \underbrace{\ind_{\{0\}\times I_{\theta_u}}}_{\in W_{t_u}} * \underbrace{\ind_{\{e_u+2N_{l+1}\}\times \mathcal{X}}}_{\in W_{e_u + 2N_{l+1}}} \\ & \in & W_{n+p+2e_u} * \ind_{\{0\}\times W} * W_{2N_{l+1}+n+q+e_u+t_u+1} \\ & \subseteq & W_{12N_{l+1}} *  \ind_{\{0\}\times W} * W_{9N_{l+1}},
\end{eqnarray*}
and the claim is proved.
\end{proof}

We return to complete the proof of Proposition \ref{prop:upper bound Ret}. By Lemma \ref{lem:every 8} and (\ref{eq:ineq N}), 
\begin{eqnarray*}
\ind_{\{0\}\times \mathcal{X}} & \in & W_{cN_{l+1}} * \ind_{\{0\}\times W} * W_{cN_{l+1}} \\ & \in & W_{cN_{l+1}} * W_{2N_{l+1}+n} * f * W_{2N_{l+1}} * W_{cN_{l+1}} \\ & \subseteq & 
W_{(c+2)N_{l+1}+n} * f * W_{(c+2)N_{l+1}} \\ & \subseteq &
W_{(c+2)\lceil Kn^\gamma \rceil + n}  * f * W_{(c+2)\lceil Kn^\gamma \rceil} \\ & \subseteq & V^{8((c+3)\lceil Kn^\gamma \rceil)} * f * V^{8(c+2)\lceil Kn^\gamma \rceil},
\end{eqnarray*}
proving that $Ret_{\mathbbm{k}[\mathfrak{G}_w]}(n) \preceq n^\gamma$.
\end{proof}

\begin{proof}[{Proof of Theorem \ref{thm:main Bavula}}]
Let $\alpha\geq 1$ be given. We may assume that $\alpha>1$ (e.g. the filter dimension of the first Weyl algebra is equal to $1$). Consider the above construction of $w$ with respect to $\gamma = \alpha$. By Lemma \ref{lem:lower bound on Ret}, we have that $Ret_{\mathbbm{k}[\mathfrak{G}_w]}(n) + n \succeq Rec_w(n)$ or, in other words, $$C \cdot Ret_{\mathbbm{k}[\mathfrak{G}_w]}(Cn) + Cn \geq Rec_w(n)$$ for some $C>0$. 
By Corollary \ref{cor:lower bound on Rec},
$$Rec_w(n) \geq 3^{-\alpha} n^\alpha$$
infinitely often.
Since we take $\alpha>1$, these two inequalities imply that for some constant $c>0$,
$$Ret_{\mathbbm{k}[\mathfrak{G}_w]}(n)\geq cn^\alpha$$
infinitely often. Therefore
$$\fdim(\mathbbm{k}[\mathfrak{G}_w])=\overline{\lim_{n\rightarrow \infty}} \frac{\log Ret_{\mathbbm{k}[\mathfrak{G}_w]}(n)}{\log n} \geq \alpha.$$
By Proposition \ref{prop:upper bound Ret},
$$Ret_{\mathbbm{k}[\mathfrak{G}_w]}(n) \preceq n^\alpha$$
so $\fdim (\mathbbm{k}[\mathfrak{G}_w]) \leq \alpha$. Finally, since $p_w(n) \sim n$ (by Proposition \ref{prop:linear complexity}), it follows by \cite[Proposition 4.5]{Nek} that the growth of $\mathbbm{k}[\mathfrak{G}_w]$ is quadratic so $\GK(\mathbbm{k}[\mathfrak{G}_w])=2$. The proof is completed.
\end{proof}

\end{document}